\documentclass[graybox, envcountsect,envcountsame]{svmult}

\usepackage{mathptmx}       % selects Times Roman as basic font
\usepackage{helvet}         % selects Helvetica as sans-serif font
\usepackage{courier}        % selects Courier as typewriter font
\usepackage{type1cm}        % activate if the above 3 fonts are
\usepackage{makeidx}         % allows index generation
\usepackage{graphicx}        % standard LaTeX graphics tool
\usepackage{multicol}        % used for the two-column index
\usepackage[bottom]{footmisc}% places footnotes at page bottom
\usepackage{changepage}% http://ctan.org/pkg/changepage
\usepackage{lipsum}% http://ctan.org/pkg/lipsum

\usepackage{amsmath,amssymb}

\makeindex             % used for the subject index
\def\Q{{Q}}

\def\OO{\mathcal{O}}

\def\ff{\frak}
\def\Spec{\mbox{\rm Spec}\,}

\def\Max{\mbox{\rm Max}\,}
\def\min{\mbox{\rm Min}\,}
\def\End{\mbox{\rm E}}

\def\cal{\mathcal}

\def\QQ{\mathcal{Q}}
\def\ff{\frak}
\def\Spec{\mbox{\rm Spec}\:}
\def\ispec{\mbox{\rm Irr}\:}
\def\imin{\mbox{\rm Min}\:}

\def\End{{E}}

\def\cal{\mathcal}

\def\X{\mathfrak{X}}
\def\U{\mathcal{U}}

\def\QQ{\mathop{\rm Quot}}

\def\PPD{{\mathbb{P}}^1_D}
\def\patch{\mbox{\rm patch}}
\def\gen{\mbox{\rm gen}}

\def\pt{\mbox{\rm pt}}
\def\inv{{\rm inv}}

\def\cl{\mbox{\rm cl}}
\def\Kr{{\rm Kr}}
\def\ev{{\rm ev}}
\def\cal{\mathcal}
\def\PPDn{{\mathbb{P}}^n_D}
\def\U{{\cal U}}
\def\V{{\cal V}}

\newcommand{\up}{{\uparrow}}
\newcommand{\down}{{\downarrow}}

\begin{document}

\title*{Topological aspects of irredundant intersections of ideals  and valuation rings}
\author{Bruce Olberding}
\institute{Bruce Olberding \at New Mexico State University, Department of Mathematical Sciences, Las Cruces, NM 88003-8001 \email{olberdin@nmsu.edu}
}
\maketitle

\abstract*{Each chapter should be preceded by an abstract (10--15 lines long)}

\abstract{An intersection of sets $A = \bigcap_{i \in I}B_i$ is irredundant if no $B_i$ can be omitted from this intersection. We develop a topological approach to irredundance by introducing a notion of a spectral representation, a spectral space whose members are sets that intersect to a given set $A$ and whose topology encodes set membership. We define a notion of a minimal representation and show that for such representations, irredundance is a topological property. We apply this approach to intersections of valuation rings and ideals. In the former case we focus on Krull-like domains and Pr\"ufer $v$-multiplication domains, and in the latter on irreducible ideals in arithmetical rings. Some of our main applications are to those rings or ideals that can be represented with a Noetherian subspace of a spectral representation. 
}

$\:$

{\noindent}{\bf Keywords} $\:$  $\bullet$ Zariski-Riemann space  $\bullet$ valuation ring $\bullet$ Krull domain $\bullet$ Pr\"ufer domain $\bullet$ Pr\"ufer $v$-multiplication ring $\bullet$ spectral space

$\:$

{\noindent}{\bf Mathematics Subject Classification (2010):} 
13F30, 13F05, 13A15

\section{Introduction}

The goal of this article is to develop a topological framework for recognizing and dealing with an irredundant infinite intersection of ideals, subrings, submodules,  even sets. While our main interest here is in the intersection of valuation rings, we include one application to the intersection of irreducible ideals in arithmetical rings to illustrate how the framework applies in a different setting. 
A key requirement for our point of view is that the objects from which the intersection is formed be drawn from  a spectral space whose topology encodes  set membership. The Zariski  topology on the set of irreducible ideals of an arithmetical  ring provides one such context, while the inverse topology on the Zariski-Riemann space of valuation rings of a field is another. Several other contexts to which our approach applies, and which we do not pursue, are given in Example~\ref{spectral exs}. 

%Because of this requirement, the only two convincing applications of the theory we are able to give at present are intersections of  (a) irreducible ideals in arithmetical rings, and (b) valuation rings of a field. However, since these are two important and well-studied, but at first glance disparate, classes of objects in Multiplicative Ideal Theory, it is hoped that having a unified approach to both topics justifies the topological framework. 

Irredundance of  intersections of valuation rings is often a consequential and special phenomenon. For example, if $F/k$ is  a finitely generated field extension of transcendence degree one with $k$ algebraically closed in $F$, and $\X$ is the set of all valuation rings containing $k$ and having quotient field $F$, then $k = \bigcap_{V \in \X}V$ and this intersection is irredundant. Thus $\X$ is  the unique representation of $k$ as an (irredundant) intersection of valuation rings in $\X$. This is a consequence of  Riemann's Theorem for projective curves and is closely related to the Strong Approximation Theorem for such curves \cite[Theorem 2.2.13]{Gol}. If, however, $F/k$ has transcendence degree $>1$, then $k$ can still be represented  by an irredundant intersection of valuation rings (albeit by very specially selected subsets of $\X$), but there exist infinitely many such representations. Such examples can be constructed along the lines of \cite[Example 6.2]{OIrr}.

In general, the existence, much less uniqueness, of an irredundant representation of a ring can only be expected under circumstances where ``few'' valuation rings are needed to represent the ring. For example, Krull domains can all be represented by an irredundant intersection of valuation rings, but this ultimately depends on the fact that they can be represented by a finite character intersection of valuation rings; see Section 5. On the other hand, if $F/k$ is a function field in more than one variable and $k$ is existentially, but not algebraically, closed in $F$, and $A$ is the intersection of all the valuation rings in $F/k$ having  residue field $k$, then no representation  $X$ of $A$ as an intersection of valuation rings contains an irredundant member; i.e., any member of $X$ can be omitted and the intersection will remain $A$; see  \cite[Theorem 4.7]{OlbH}.  This last example is even a Pr\"ufer domain and hence has the property that every valuation ring between $A$ and its quotient field is a localization of $A$. Thus even for classes of rings whose valuation theory is explicitly given by their prime spectra, intersections of valuation rings can behave in complicated ways.  

In Section 3 we develop a topological approach to these issues for intersections of sets, where the sets themselves can be viewed as points in a spectral space. The prime ideals of a ring or the valuation rings of a field comprise such sets when viewed with the appropriate topologies, but also so do the irreducible ideals in an arithmetical ring. 
Throughout this article we are particularly interested in Noetherian spectral spaces, and in Section 2 we work out some of the properties of these spaces when viewed under the inverse or patch topologies. (These topologies are reviewed in Section 2.) Krull domains, and  generalizations of these rings of  classical interest,  can be represented by intersections of valuation rings drawn from a Noetherian subspace of a spectral space, and we apply the  results from Sections 2 and 3 in Sections 5 and 6 to intersections of valuation rings from a Noetherian subspace of the Zariski-Riemann space of a field.

Sections 4-7 contain the main applications of the article. Section 4 applies the abstract setting of spectral representations to the Zariski-Riemann space of a field. This section recasts the abstract approach in Section 3 into  a topological framework for working with irredundance in intersections of valuation rings.
 Section 5 specializes the discussion to the Krull-like  rings of classical interest and recaptures the representation theorems for these rings. 
  A feature throughout Sections 4 and 5 that is afforded by the abstract approach of spectral representations  is that intersections of valuation rings can be considered relative to a subset of the ambient field. The motivation for this comes from the articles 
 \cite{AHE, HOOlb, LTOlber, ONoeth,OT}.  In these studies, one considers integrally closed domains  $A$ between a given domain and overring, e.g., between ${\mathbb{Z}}[T]$ and ${\mathbb{Q}}[T]$. In such cases $A$ is an intersection of ${\mathbb{Q}}[T]$ and valuation rings not containing ${\mathbb{Q}}[T]$. Since ${\mathbb{Q}}[T]$ can be viewed as always present in these representations, it is helpful then  to consider representations of a ring $A$ of the form $A = (\bigcap_{V \in X}V) \cap C$, where $C$ is a fixed ring.  The approach provided by Section 3 makes it easy to incorporate a fixed member $C$ of the representation into such a picture, regardless of whether $C$ is a ring or simply a set.  %This is in much the same spirit as Heinzer and Ohm's generalization in \cite{HOOlb} of Krull domains relative to some overring $C$. 

 The already well-understood theory of irredundance for Pr\"ufer domains also can be recovered from our framework, and this is done in Section 6 in the more general setting of Pr\"ufer $v$-multiplication domains.  We consider existence and uniqueness for irredundant representations of such domains, with special emphasis on the case in which the space of $t$-maximal ideals is Noetherian. When restricted to a Pr\"ufer domain $A$, these results specialize to a topological characterization of the property that every overring of  $A$ is an irredundant intersection of the valuation rings that are minimal over it.

In order to help justify the generality of the approach Section 3, we show in Section 7 how the topological framework can be applied to the study of irredundant intersections of irreducible ideals in arithmetical rings. We show in particular how  intersection decomposition results involving such ideals can be recovered from our point of view. 
 This section is  independent of the valuation-theoretic Sections 4--6.  
 
 I thank the referee for helpful comments that improved the clarity of some of the arguments.

\section{Spectral spaces}

A {\it spectral space} is a $T_0$ topological space having (a)  a basis of quasicompact open sets closed under finite intersections, and (b) the property that every irreducible closed subset has a unique generic point, i.e., a point whose closure is the irreducible closed set. By a theorem of Hochster \cite[Corollary, p.~45]{Hoc}, 
a topological space $X$ is spectral if and only if $X$ is homeomorphic to the prime spectrum of a ring. In the setting of this paper, it is mostly the topological features of spectral spaces  that are needed rather than the connection with prime spectra of rings. 

A spectral space $X$ admits two other well-studied topologies that are useful in our context. The {\it inverse topology} on $X$ has as a basis of closed sets the subsets of $X$ that are quasicompact and open in the spectral topology.  By an {\it inverse closed} subset of $X$ we mean a subset that is closed in the inverse topology. The {\it patch topology} has as a basis of open sets the sets of the form $U \cup V$, where $U$ is open and quasicompact in the spectral topology and $V$ is the complement of a quasicompact open set. These basic open sets are also closed, so that the patch topology is zero-dimensional and Hausdorff. A {\it patch} in $X$ is a set that is closed in the patch topology.  In this section we denote the closure of a subset $Y$ of $X$ in the spectral topology as $\overline{Y}$, and  the closure of $Y$ in the patch topology as $\widetilde{Y}$. 

The patch topology refines both the spectral and inverse topologies. This can be made more precise using the {\it specialization order} of the spectral topology: If $x,y \in X$, then $x \leq y$ if and only if $y \in \overline{{\{x\}}}$ in the spectral topology. With this order in mind, we define for  $Y \subseteq X$, 
\begin{description}[(3)]
\item[] $\up Y     = \{x \in X:x \geq y$ some $y \in Y\}$ \: and \: $\down Y =\{x \in X:x \leq y$ some $y \in Y\}$,  
\item[\:] $\min Y:=\{y \in Y:y$ is minimal in $Y$ with respect to $\leq \}$, 
\item[\:] $\Max Y:=\{y \in Y:y$ is maximal in $Y$ with respect to $\leq \}$. \end{description}
%In a spectral space $X$, for every element $x \in X$, there exists an element $m \in X$ that is minimal with respect to the specialization order and for which $m \leq x$. We denote the collection of minimal elements as
%  \begin{center}$\min Y:=\{y \in Y:y$ is minimal in $Y$ under the specialization order$\}$. \end{center}
%Similarly, we denote the maximal elements of $Y$ as  \begin{center}$\Max Y:=\{y \in Y:y$ is maximal in $Y$ under the specialization order$\}$. \end{center}

\begin{proposition}
 \label{Hoc prop} Let $X$ be a spectral space with specialization order $\leq$.  Then 
\begin{description}[(3)]
\item[{\em (1)}] $X$ with the inverse topology is a spectral space whose specialization order is the reverse of that of $(X,\leq)$. 
\item[{\em (2)}] $X$ with the patch topology is a spectral space, and in particular a compact Hausdorff zero-dimensional space. 
\item[{\em (3)}]   For each $Y \subseteq X$,  $\overline{Y} = \up(\widetilde{Y})$  and the closure of $Y$ in the inverse topology is $ \down(\widetilde{Y})$. 
\item[{\em (4)}] 
If $Y$ is a patch in $X$, then the following statements hold. 
\begin{description}[(a)]
\item[{\em (a)}]   $Y$ is spectral in the subspace topology. 
\item[{\em (b)}]  For each $y \in Y$ there exists $m \in \min Y$ with $m \leq y$.   
\item[{\em (c)}] The patch and spectral topologies agree on $\min Y$.
%\item[{\em (6)}] If $Z$ is also a patch in $X$, then $\up Y = \up Z$ if and only if $\min Y = \min Z$.  
\end{description}

\end{description}
\end{proposition}  

\begin{proof} Statement (1) can be found in \cite[Proposition 8]{Hoc}; statement (2) can be deduced from \cite[Section 2]{Hoc}.  Statement (3) is a consequence of \cite[Corollary, p.~45]{Hoc} and (1).   
 Statement (4)(a) follows from \cite[Proposition 9]{Hoc}.  Statement (4)(b) now 
   follows from (a), since a spectral space has minimal elements. Finally, the spectral and patch topologies agree on the set of minimal elements of  a spectral space \cite[Corollary 2.6]{ST}, so (4)(c) follows from (4)(a).  
% For statement (6), note that by (2), $\up Y = \overine{Y}$ and $\up$
\qed\end{proof}

We give now a list of examples of some  spectral spaces in our context. We only use a few of these examples in what follows, but the intersection representation theory developed in the next section applies to all of them. As we indicate, several of  these examples have appeared in the literature before, but with different proofs than what we give here. Our approach is inspired by a theorem of Hochster \cite[Proposition 9]{Hoc} that a topological space is spectral if and only if it is homeomorphic to a patch closed subset of a power set endowed with the hull-kernel topology. Interestingly, inspection of Zariski and Samuel's proof in \cite{ZS} that the Zariski-Riemann space $\X$ of a field is quasicompact shows that although their work predated the notion of spectral spaces, what is proved there is that $\X$ is a patch closed subset of a certain spectral space, and hence from their argument can be deduced the fact that $\X$ is spectral. 

To formalize the setting of the example, 
let $S$ be a set. We denote by ${\mathbf{2}}^S$ the power set of $S$ endowed with the hull-kernel topology having  as an open basis  the sets of the form $\U(F) := \{B \subseteq S:F \not \subseteq B\}$, where  $F$ is a finite subset of $S$.   The complement of $\U(F)$ is denoted $\V(F)$; i.e., $\V(F) = \{V \subseteq S:F \subseteq B\}$. 
Then the sets $\U(F)$ are quasicompact and 
 ${\mathbf{2}}^S$ is a spectral space; cf.~\cite[Theorem 8 and~Proposition 9]{Hoc}.  Thus by Proposition~\ref{Hoc prop}, to show that a collection $X$ of subsets of $S$ is a spectral space in the subspace topology, it is enough to show that $X$ is patch closed in $S$. Specifically, what must be shown is that $X$ is an intersection of sets of the form $\V(F_1) \cup \cdots \cup \V(F_n) \cup \U(G)$, where $F_1,\ldots,F_n,G$ are  finite subsets of $S$.  
This is done in each case by encoding the question of whether a given subset of $S$ satisfies a first-order property in the relevant language into an assertion about membership in a  set of the form $\V(F_1) \cup \cdots \cup \V(F_n) \cup \U(G)$. This amounts in most cases to rewriting a statement of the form ``$p \rightarrow q$'' as ``(not $p$) or $q$.''  Because the goal is to produce patch closed subsets,  statements involving universal quantifiers (which translate into intersections) are  more amenable to this approach than statements involving existential quantifiers (which translate into infinite unions).  

To clarify terminology, when $R$ is a ring, the {\it Zariski topology} on a collection $X$ of ideals of $R$ is the hull-kernel topology defined above; i.e., it is simply the subspace topology on $X$ inherited from ${\mathbf{2}}^R$. This agrees with the usual notion of the Zariski topology on $\Spec R$.   However, when $S$ is a ring and and $X$ is a collection of subrings of $S$, then the {\it Zariski topology} on $X$ is the inverse of the hull-kernel topology; i.e., it has an open basis consisting of sets of the form $\V(G)$, where $G$ is a finite subset of $S$. Despite the discrepancy, it is  natural to maintain it in light of  the fact that 
 when $R$ is a subring of  a field $F$,  then with these definitions,  $\Spec R$  with the Zariski topology is homeomorphic to the space $\{R_P:P \in \Spec R\}$ of subrings of $F$ with the Zariski topology. This discrepancy, which is due to Zariski, also allows for an 
identification between projective models and projective schemes; cf.~\cite{ZS} for the notion of a projective model.

\begin{example} \label{spectral exs} 
(1)  {\it The  set of all proper ideals of a ring $R$ is a spectral space in the Zariski topology.} 
The set of proper ideals in $R$ is precisely the patch closed subset of  ${\mathbf{2}}^R$ given by
$$X_1 = \U(1) \cap (\bigcap_{a,b \in R} \U(a,b) \cup \V(a+b)) \: \cap \:  (\bigcap_{a,r \in R} \U(a) \cup \V(ra)).$$ 

(2) {\it If $R$ is a ring, the set of  all  submodules of an $R$-module is a spectral space in the Zariski topology.}  An easy modification of (1) shows this to be the case. 

\smallskip

(3) {\it The set  of all radical ideals of a ring $R$ is a spectral space in the Zariski topology.} The set of radical ideals  is precisely the patch closed subset of  ${\mathbf{2}}^R$ given by
$$X_3 = X_1 \cap (\bigcap_{a \in R,n>0} (\U(a^n) \cup \V(a))).$$ 

(4) {\it If $R $ is a ring such that $aR \cap bR$ is a finitely generated ideal of $R$ for all $a,b \in R$, then the set of all proper strongly irreducible ideals is a spectral space in  the Zariski topology.}  Recall that an ideal $I$ of $R$ is strongly irreducible if whenever $J \cap K \subseteq I$, then $J \subseteq I$ or $K \subseteq I$; equivalently, $I$ is strongly irreducible if and only if whenever $a,b \in R$ and $aR \cap bR \subseteq I$, it must be that $a \in I$ or $b \in I$. Thus the set of strongly irreducible proper ideals in $R$ is  given by
 $$X_4 = X_1 \cap (\bigcap_{a,b\in R} (\U(aR \cap bR) \cup \V(aR) \cup \V(bR))).$$
By assumption, for each $a,b \in R$,  $aR \cap bR$ is a finitely generated ideal of $R$, so the set $\U(aR \cap bR)$ is quasicompact and open. Therefore, $X_4$ is a patch closed subset of ${\mathbf{2}}^R$.  This example will be used in Section 7. 
\smallskip

(5) (Finocchiaro \cite[Proposition 3.5]{Fin})  {\it 
Let $R \subseteq S$ be an extension of rings. The set of rings between $R$ and $S$ with the Zariski topology is a spectral space.} The set of rings between $R$ and $S$ is given by the patch closed set
$$X_5 = (\bigcap_{r \in R} \V(r)) \cap (\bigcap_{a,b \in S}\U(a,b) \cup \V(a+b,ab)).$$ The Zariski topology on $X_5$ is the inverse topology of the subspace topology on $X_5$ inherited from  ${\mathbf{2}}^S$, so by Proposition~\ref{Hoc prop}, $X_5$ is spectral in the Zariski topology.  
\smallskip

(6)  (Finocchiaro \cite[Proposition 3.6]{Fin})  {\it Let $R \subseteq S$ be an extension of rings. The  set of all integrally closed rings between $R$ and $S$ with the Zariski topology  is a spectral space.} Let $\cal M$ denote the set of monic polynomials in $S[T]$, and for each $f \in {\cal M}$, let $c(f)$ denote the set of coefficients of $f$.  
 The set of integrally closed rings between $R$ and $S$ is given by the patch closed set
$$X_6= X_5 \cap \left(\bigcap_{s \in S} \left(\bigcap_{f \in {\cal M},f(s) = 0}\U(c(f)) \cup \V(s)\right)\right).$$ As in (5), this implies that $X_7$ is spectral in the Zariski topology. 
\smallskip

(7) (Finocchiaro-Fontana-Spirito \cite[Corollary 2.14]{FFS}) {\it Let $R$ be a subring of a field $F$.  The set of all local rings between $R$ and $F$ with the Zariski topology is a spectral space.}  A ring $A$ between $R$ and $F$ is local if whenever $a,b$ are nonzero elements of $R$ with $1/(a+b) \in R$, we have  $1/a \in R$ or $1/b \in R$. Thus the set of all local rings between $R$ and $F$ is given by the patch closed subset
$$X_7 = X_5 \cap \left(\bigcap_{0 \ne a,b \in F}\U(a,b,1/(a+b)) \cup \V(1/a) \cup \V(1/b)\right).$$  As in (5), this implies that $X_7$ is spectral in the Zariski topology. 
\smallskip

(8) {\it  Let $A$ be a subring of a field $F$. The set of all valuation rings containing $A$ and having quotient field $F$ is a spectral space in the Zariski topology.} This has been proved by a number of authors; see \cite{FFL,OZR} for discussion and references regarding this result. 
 A subring $V$ between $A$ and $F$ is a valuation ring with quotient field $F$ if and only if for all $0 \ne q \in F$, $q \in V$ or $q^{-1} \in V$.  
Thus with $R=A$ and $S = F$, we use the set $X_5$ from (5) to obtain the set of valuation rings of $F$ containing $A$ as the patch closed subset of ${\mathbf{2}}^F$ given by 
$$X_8 = X_5 \cap \bigcap_{0 \ne q \in F}\V(q) \cup \V(q^{-1}).$$ As in (5), this implies that $X_8$ is spectral in the Zariski topology. 
%\smallskip

%(9) {\it Let $R$ be a subring of a field $F$. The set of all rings $A$ between $R$ and $F$ such that $A$ is not a Pr\"ufer domain with quotient field $F$ is a spectral space with the Zariski topology.} 
%A ring $A$ between $R$ and $F$ is a Pr\"ufer domain with quotient field $F$ if and only if for all $0 \ne q \in F$, $(A \cap qA) + (A \cap q^{-1}A)  = A$. (This follows for example from local verification using the fact that each localization of $A$ at a maximal ideal is a valuation domain.) Thus a ring $A$ between $R$ and $F$ is not a Pr\"ufer domain with quotient field $F$ if and only if whenever $a,b \in A$ and $0 \ne q\in F$ such that $aq,bq^{-1} \in A$, it must be that $1/(aq+bq^{-1}) \not \in A$.  
%Therefore,  The set of all rings $A$ between $R$ and $F$ such that $A$ is not a Pr\"ufer domain with quotient field $F$ is given by
%$$X_9= X_5 \cap \left(\bigcap_{a,b,q\in F, q \ne 0} \U(a,b,aq,bq^{-1},1/(aq+bq^{-1})\right).$$
%As in (5), this implies that $X_9$ is spectral in the Zariski topology.  Moreover, in the Zariski topology, $X_9$ is a closed subset of $X_5$. Thus, in the Zariski topology, the set of Pr\"ufer domains with quotient field $F$ is an open subset of the space of all rings between $A$ and $F$.  
\end{example}

For the remainder of the section we focus on Noetherian spectral spaces, since these play a central role in later sections. 
 A topological space is {\it Noetherian} if its open sets satisfy the ascending chain condition. 
 Rush and Wallace \cite[Proposition 1.1 and Corollary 1.3]{RW}  have shown that a collection $Y$ of prime ideals of a ring $R$ is a Noetherian subspace of $\Spec R$ if and only if for each prime ideal $P$ of $R$, there is a finitely generated ideal $I \subseteq P$ such that every prime ideal in $Y$ containing $I$ contains also $P$. Since every spectral space can be realized as $\Spec R$ for some ring $R$, we may restate this topologically in the following form.   
  
  \begin{lemma} {\em (Rush and Wallace)}  \label{RW lemma} Let $X$ be a spectral space, and let $Y$ be a subspace of $X$. Then $Y$ is Noetherian  if and only if for each 
  irreducible closed subset $C$ of $X$,  $Y \cap C  = Y \cap C'$ for some closed subset $C' \supseteq C$  such that  $X \setminus C'$ is quasicompact. 
\end{lemma}

In later sections we focus on spectral spaces $X$ 
 in which the set of maximal elements under the specialization order $\leq $ of $X$ is a Noetherian space. The spectral spaces in our applications have  the additional property that $(X,\leq)$ is a tree. In this case, as we show in Theorem~\ref{Noetherian prop},  the Noetherian property for the maximal elements descends to subsets consisting of incomparable elements.

\begin{lemma} \label{technical Noetherian} Let $X$ be a spectral space whose specialization order $\leq$ is a tree. Suppose that
 $\Max X$ is a Noetherian space.  
% \begin{description}[(2)]  
% \item[{\em (1)}] If $C$ is a nonempty  closed subset of $X$, then $C = C_1 \cup \cdots \cup C_n$, where the $C_i$  are  irreducible closed subsets  of $X$  such that  $C \setminus C_i$ is  quasicompact. 
%\item[{\em (2)}]  
Then a  subspace $Y$ of $X$ is Noetherian if and only if $(Y,\leq)$ satisfies the ascending chain condition.
%\item[{\em (3)}] Every nonempty  subset of $Y$ consisting of elements that are incomparable under $\leq$ is discrete in the inverse and patch topologies. 
%\end{description}
\end{lemma}

\begin{proof}
 If $Y$ is Noetherian, then the closed subsets of $Y$ satisfy the descending chain condition, so $(Y,\leq)$ satisfies the ascending chain condition. Conversely, 
 suppose that $(Y,\leq)$ satisfies ACC. Let  $C$ be an irreducible closed subset of $X$. By Lemma~\ref{RW lemma}, to prove that $Y$ is Noetherian, it suffices to show that there exists a closed subset $C' \supseteq C$  such that $Y \cap C = Y \cap C'$ and $X \setminus C'$ is quasicompact.  By Lemma~\ref{RW lemma}, there exists a closed subset $C_1 \supseteq C$  such that $C \cap \Max X = C_1 \cap \Max X$ and $X \setminus C_1$ is quasicompact.  
  Since $C$ is irreducible, there is $c \in C$ such that $C = \{x \in X:c \leq x\}$. 
%As in the proof of (1), the fact that $\Max Y$ is Noetherian implies there exists a closed subset $C_1$ of $Y$ such that $C \cap \Max Y = C_1 \cap \Max Y$ and $Y \setminus C_1$ is quasicompact in $Y$.  
 Let $D = \bigcap_{y < c,y \in Y}\overline{\{y\}}$.  Since 
 $(X,\leq)$ is a tree and $(Y,\leq)$ satisfies ACC, $C$ is a proper subset of $D$. 
 Thus since the quasicompact open subsets of $X$ form a basis for $X$, there is a closed set $C_2$ such that $C \subseteq C_2$, $D \not \subseteq C_2$ and $X \setminus C_2$ is quasicompact. 
We claim that $Y \cap C = Y \cap C_1 \cap C_2$.  
 The containment ``$\subseteq$'' is clear since $C \subseteq C_1 \cap C_2$. Let $y \in Y \cap C_1 \cap C_2$. Then there exists $m \in \Max X$ such that $y \leq m$.  Thus $m \in \Max X \cap C_1 \cap C_2 \subseteq C$, so that $c \leq m$. Since $(X,\leq)$ is a tree and $y,c\leq m$, it must be that $y < c$ or $c \leq y$. If $y< c$, then $C \subsetneq D \subseteq \overline{\{y\}}$. However, since $y \in C_2$, this forces $D \subseteq C_2$, a contradiction. Thus $c \leq y$, and hence $y \in C$. This shows that $Y \cap C = Y \cap C_1 \cap C_2$. Finally, since $X\setminus C_1$ and $X \setminus C_2$ are quasicompact, so is their union $X \setminus (C_1 \cap C_2)$. Thus with $C' = C_1 \cap C_2$  the claim is proved. 
 \qed\end{proof}

\begin{theorem} \label{Noetherian prop} Let $X$ be a spectral space whose specialization order $\leq$ is a tree. Then  $\Max X $ is a Noetherian space in the spectral topology if and only if every
   subset of $X$ consisting of elements that are incomparable under $\leq$ is discrete in the inverse   topology.

%\item[{\em (3)}]  Every 
%   subset of $X$ consisting of elements that are incomparable under $\leq$ is discrete in the patch   topology. 

%   \item[{\em (4)}]  For every closed subset $C$ of $X$, $\min C$ is discrete in the inverse (equivalently, patch) topology.
 %  \end{description}
\end{theorem} 

\begin{proof}  Suppose $\Max X$ is a Noetherian space. 
 Let $Y$ be a nonempty subset of $X$ whose elements are incomparable under $\leq$.  
   Then  by Lemma~\ref{technical Noetherian}, $Y$ is a Noetherian space. Let $y \in Y$.
   Since the elements of $Y$ are incomparable,   $Y \setminus \{y\}$ is open in $Y$. In a Noetherian space, open sets are quasicompact, so $Y \setminus \{y\}$ is 
   inverse closed in $Y$, which proves that $y$ is isolated in the inverse topology on $Y$.
 
% (2) $\Rightarrow$ (3) This is clear since the patch topology refines the inverse topology.
 
% (3) $\Rightarrow$ (4) This is clear since the inverse and patch topology agree on the minimal elements of  a closed subset of a spectral space {\bf [ref]}. 

Conversely, suppose that every subspace of $X$ consisting of incomparable elements   is discrete in the inverse  topology. 
To prove that $\Max X$ is Noetherian, it suffices by Lemma~\ref{RW lemma} to show that for each irreducible closed subset $C$ of $X$ there exists a closed set $C' \supseteq C$ such that $C \cap \Max X = C' \cap \Max X$ and $X \setminus C'$ is quasicompact. Let $C$ be an irreducible closed subset of $X$,
%it suffices to show that every open subset of $\Max X$ is quasicompact.  Let $C$ be a closed subset of $X$. We show that $\Max X \setminus C$ is quasicompact. Since $C$ is spectral, $\min(C)$ is quasicompact in the inverse topology (because it is consists of the elements that are maximal in the inverse specialization order), and by assumption $\min(C)$ is also discrete in the inverse topology. Thus $\min(C)$ is finite, so that $C = C_1 \cup \cdots \cup C_n$ for irreducible closed subsets $C_i$ of $X$.  
%Now  $\Max X \setminus C = (\Max X \setminus C_1) \cap \cdots \cap (\Max X \setminus C_n)$. 
%
%Since $X$ is a $T_0$ space, the set of maximal elements of any quasicompact subset of $X$ is quasicompact. Since each $C_i$ is irreducible and $(X,\leq)$ is a tree, $\Max(X \setminus C_i) = \Max X \setminus C_i$.  
%Thus to prove that $
and let $c \in C$ such that $C = \{x \in X:c \leq x\}$. % Let $E = \Max X \setminus C$.  
By assumption, $\{c\} \cup ((\Max X) \setminus C)$ is discrete in the inverse topology since the elements in this set are incomparable. Thus there exists a quasicompact open subset $U$ of $X$ such that $(\Max X) \setminus C \subseteq U$ and $c \not \in U$. 
Since $c \not \in U$ and $U$, being open, has the property that $U = \down U$, it must be that $C \cap U = \emptyset$. 
Thus $U \subseteq X \setminus C$, so that $(\Max X) \cap U \subseteq (\Max X) \setminus C$. Since also $(\Max X) \setminus C \subseteq U$, we conclude  $(\Max X) \setminus C = (\Max X) \cap U$.  
 Thus with  $C' = X \setminus U$, we have    $(\Max X) \cap C = (\Max X) \cap C'$ and $X \setminus C'$ is quasicompact. 
Since also $C \subseteq X \setminus U = C'$,  the claim is proved. 
\qed\end{proof}

\begin{corollary} \label{Noetherian min}  Let $X$ be a spectral space whose specialization order $\leq$ is a tree. If  $\Max X $ is a Noetherian space, then $\min C$ is finite for each nonempty closed subset $C$ of $X$.
\end{corollary}

\begin{proof}
Let $C$ be a nonempty closed subset of $X$. By Proposition~\ref{Hoc prop}(1),  $\min C$  consists of the elements that are maximal with respect to the specialization order in the inverse topology. The set of maximal elements of a spectral space is a quasicompact subspace, so 
$\min C$ is quasicompact in the inverse topology. Thus since by Theorem~\ref{Noetherian prop}, $\min C$ is discrete in the inverse topology, $\min C$ is finite. \qed
\end{proof}

\section{Spectral representations}

Throughout this section $A,C$ and $D$ are nonempty sets with $A \subsetneq C \subseteq D$. We do not assume the presence of any algebraic structure on these sets. We work under the assumption that $A$ can be represented as an intersection of $C$ with sets $B$ between $A$ and $D$ such that the sets $B$ are points in a spectral space $X$  whose specialization order is  compatible with  set inclusion. The set $C$ can be viewed as a fixed component in the intersection (e.g., the case $C = D$ is often an interesting choice). The goal then is to make ``efficient'' choices from $X$ to represent $A$.  
We formalize some of this with the following definition.  

\begin{definition} \label{long def}
Let   $X$ be a collection of subsets of $D$, and assume that $X$ is a {\it $C$-representation of $A$}, %(or simply $C/A$ when the choice of $D$ is clear) 
meaning that $A =  (\bigcap_{B \in X}B) \cap C$.     
\begin{description}[(1)]
%\item[{(1)}] We say that $X$ is a {\it $C$-representation of $A$} %(or simply $C/A$ when the choice of $D$ is clear) 
%if $A = C \cap (\bigcap_{B \in X}B)$. 

\item[{(1)}] 
For each $F \subseteq D$, let $\V(F) = \{B \in X:F \subseteq B\}$ and $\U(F) = \{B \in X:F \not \subseteq  B\}$.  We say the $C$-representation   $X$ is   {\it spectral}  if $X$ is a spectral space and $\{\U(d):d \in D\}$ is a subbasis for $X$ consisting of quasicompact open sets.  Note that this choice of subbasis assures that the specialization order agrees with the partial order given by set inclusion. 
\end{description}
Now assume that $X$ is a spectral $C$-representation of $A$. 
\begin{description}[(1)]
\item[(2)] Let $Z \subseteq X$ be a $C$-representation of $A$, and let  $B \in Z$.   
Then $B$ is {\it irredundant} in $Z$ if  $Z \setminus \{B\}$ is not a $C$-representation of $A$; 
%$B$ is 
 %{\it strongly irredundant} in $Z$  if  $B$ is irredundant in $Z$ and
% whenever $F$ is a nonempty closed subset of $ \V(B)$ such that $(Z \setminus \{B\})\cup F  $ is a $C$-representation of $A$, then $B = \bigcap_{B' \in F}B'$; 
$B$ is 
 {\it strongly irredundant} in $Z$  if  the only  closed subset $Y$ of $ \V(B)$ such that $(Z \setminus \{B\})\cup Y  $ is a $C$-representation of $A$ is  $Y = \V(B)$;
 %
% $(Z \cup F) \setminus  \{B\}$  is not a $C$-representation of $A$
$B$   is {\it tightly irredundant} in $Z$ if %in this same intersection 
$(Z \cup \V(B)) \setminus \{B\}$ is not a $C$-representation of $A$.  

%\end{description}
%\end{definition}

%\begin{definition} Let $A \subseteq C \subseteq D$ be sets, and let $X$ be a spectral representation of $C/A$. The {\it patch topology} on $X$ is the topology with subbasis of open sets $$\{\U(d):d \in D\} \cup \{\V(d)^c:d \in D\}.$$  
%\end{definition} 

%The patch topology on  a representation space is zero-dimensional (i.e., has a basis consisting of sets that are both closed and open), compact and Hausdorff. 

%\begin{definition}
%Let  $X$ be a spectral $C$-representation of $A$.  
%\begin{description}[(3)] 
\item[(3)]   A  closed subset $Y$ of $X$ (resp., a patch) that is minimal with respect to set inclusion among closed (resp., patch) $C$-representations of $A$ in $X$ is 
a
{\it minimal closed (resp., patch) $C$-representation  of $A$ in $X$}.

\item[(4)]  A subspace  of $X$ of the form $\min Y$ for some minimal closed $C$-representation $Y$ of $A$  is 
 a {\it minimal $C$-representation of $A$}.

%\item[(4)] A patch $Y$ in the spectral $C$-representation $X$ that is a $C$-representation of $A$ is a {\it minimal patch $C$-representation} if whenever $Z \subseteq Y$ is a patch that is a $C$-representation of $A$, we have $Z = Y$.   
\end{description}  
\end{definition}

The notions of strong and tight irredundance become much clearer in the settings of Sections 4-7; see the discussion after Definition 4.1.

Observe that in (4), since the specialization order on $X$ agrees with the partial order given by set inclusion among the members of $X$, $\min Y$ is also the minimal elements of $Y$ with respect to set inclusion.

\begin{lemma}  \label{minimal exist} Every spectral $C$-representation  of $A$ contains a minimal closed $C$-representation of $A$ and a minimal patch $C$-representation of $A$.  
\end{lemma}

\begin{proof} %let $Y$ be a subset of $Z$, 
Let $X$ be a spectral $C$-representation  of $A$, and let ${\mathcal{F}}$ be the set of closed $C$-representations of $A$  in $X$. 
 Then ${\cal{F}}$ is nonempty since $X \in {\cal{F}}$.  Let $\{Y_\alpha\}$ be a chain of elements in ${\cal F}$, and let $Y = \bigcap_{\alpha} Y_\alpha$. As an intersection of 
closed subsets, $Y$ is  closed. We claim that $Y$ is a $C$-representation of $A$.   Clearly, $A \subseteq (\bigcap_{B \in Y}B) \cap C$.  Let $d \in (\bigcap_{B \in Y}B) \cap C$.  Then $\bigcap_{\alpha} Y_\alpha = Y \subseteq \V(d)$, and hence $\U(d) \subseteq \bigcup_{\alpha}Y^c_\alpha$, where $Y^c_\alpha = X \setminus Y_\alpha$.  Since $\U(d)$ is quasicompact and each  $Y^c_\alpha$ is open, the fact that the $Y_\alpha$ form a chain under inclusion implies  that $\U(d) \subseteq Y^c_\alpha$ for some $\alpha$. For this choice of $\alpha$, $ Y_\alpha \subseteq \V(d)$, which since $Y_\alpha$ is a $C$-representation of $A$ implies that $d \in (\bigcap_{B \in Y_\alpha}B) \cap C= A$.  
  Therefore, $A =
(\bigcap_{B \in Y}B) \cap C$, which shows that $Y \in {\cal F}$.   
By  Zorn's Lemma, ${\mathcal{F}}$ contains minimal elements.  Since the patch topology is spectral by Proposition~\ref{Hoc prop}(2), the final statement follows from the first.   
\qed \end{proof}

\begin{lemma}\label{ti} \label{tfae} Let $X$ be a spectral  $C$-representation of $A$, let $Z \subseteq X$ be a $C$-representation of $A$ and let $B \in Z$.   
 Then
\begin{description}[(1)]
\item[{\em (1)}] $B$ 
 is irredundant in $Z$ if and only if $B$ is irredundant in $\widetilde{Z}$.

\item[{\em (2)}] 
 $B$ is tightly irredundant in $Z$ if and only if $B$ is irredundant in $\overline{Z}$.
\end{description}  
 Moreover, if  $B$  is irredundant in $Z$, then $B$ is isolated in the spectral and patch subspace topologies on $Z$.

\end{lemma}

\begin{proof}
(1) Suppose that $B$ is irredundant in $Z$. Then there exists $d \in D$ such that $d \not \in B$  but $d$ is  in every other set that is in $Z$. 
Thus  $Z \subseteq \V(d) \cup \{B\}$, and since $\V(d)$ and $\{B\}$ are patches in $X$, we have $\widetilde{Z} \subseteq \V(d) \cup \{B\}$. Hence $\widetilde{Z} \setminus \{B\} \subseteq \V(d)$. Since  $d \not \in B$, this implies that $B$ is irredundant in $\widetilde{Z}$.  
 The converse is clear since $Z \subseteq \widetilde{Z}$.

(2)
Suppose that $V$ is tightly irredundant in $Z$. Then there exists $d \in D$ such that $d$ is not in $B$ but $d$ is in every any other  set in $Z \cup \V(B)$. 
 Thus $(Z \cup \V(B)) \setminus \{B\} \subseteq \V(d)$, which implies that $\V(d) \cup \V(B) = \V(d) \cup \{B\}$. 
 Since  $\V(d) \cup \V(B)$ is  closed, we have $\overline{Z} \subseteq \V(d) \cup \V(B) = \V(d) \cup \{B\}$.     Therefore, $\overline{Z} \setminus \{B\} \subseteq \V(d)$, and hence   $d$ is in every set in $\overline{Z} \setminus \{B\}$. Since  $d \not \in  B$, 
we conclude that  $B$ is irredundant in $\overline{Z}$.  Conversely, if $B$ is irredundant in $\overline{Z}$, then since $\V(B) \subseteq \overline{Z}$, it follows that $B$ is tightly irredundant in $Z$.  

It remains to prove the last statement of the lemma. Suppose that $B$ is irredundant in $Z$.
 Let $Z' = Z \setminus \{B\}$. Then there exists $c \in (\bigcap_{B' \in Z'}B') \cap C$ with $c \not \in B$.  Since the set $\V(c)$ 
is closed in the spectral and patch topologies and  this set contains $Z'$ but not $B$, we have that $B$ is an   
isolated point in the spectral and patch subspace topologies on $Z$. 
\qed
\end{proof}

\begin{proposition}  \label{count} Let $X$ be a spectral $C$-representation of $A$. If $Z\subseteq X$ is a tightly irredundant $C$-representation of $A$, then $Z$ is contained in a minimal $C$-representation of $A$. Thus 
the number of minimal $C$-representations of $A$ is greater than the number of  tightly irredundant $C$-representations of $A$. 
\end{proposition}  

\begin{proof} By Lemma~\ref{ti}, the members of $Z$ are irredundant in the $C$-represent\-a\-tion $\overline{Z}$ of $A$. 
 By Lemma~\ref{minimal exist}, there exists a minimal $C$-representation $Z_1$ contained in $\overline{Z}$.  Since the members of $Z$ are irredundant in $\overline{Z}$, it must be that $Z \subseteq Z_1$.  To prove the last claim of the proposition,
 it suffices to show that distinct irredundant $C$-representations of $A$ are contained in distinct minimal $C$-representations of $A$. Suppose $Y \subseteq X$ is another tightly irredundant $C$-representation of $A$ with $Y \ne Z$. Then the members of  $Y$ are irredundant in a minimal $C$-representation $Y_1$ of $A$.   
  If $Y_1 = Z_1$, then the members of  $Y$ and $Z$ are irredundant in $Y_1$, which, since $Y$ and $Z$ are $C$-representations of $A$, implies that $Y  = Z$, a contradiction that implies $Y_1 \ne Z_1$.  
\qed
\end{proof} 

Example~\ref{three} shows that neither tightly irredundant nor minimal representations need  be unique.

\begin{lemma}  \label{remove}
Let  $X$ be a spectral $C$-representation  of $A$, and let $Z$ be a minimal  $C$-representation  of $A$ in $X$. Then
% \begin{description}
% \item[{\em (1)}]  
$\overline{Z}$ is a minimal  closed $C$-representation of $A$, 
% \item[{\em (2)}] 
  $\widetilde{Z}$ is  a minimal  patch  $C$-representation of $A$, % and $\inv(Z)$ is a minimal inverse closed $R$-representation of $A$.  
%\item[{\em (3)}]  
$Z = \min \: \widetilde{Z}  = \min \: \overline{Z}$   and  $\overline{Z} = \up Z$.  
%\end{description}
\end{lemma}

\begin{proof} 
%We first observe that if $Y$ is a nonempty patch in $X$, then $Y \subseteq {\up(\min Y)}$, and if also  $Y$ is closed, then $Y = \up (\min Y)$. For suppose $Y$ is a nonempty patch.  Then by Proposition~\ref{Hoc prop}(2), the subspace  topology on $Y$  is spectral, and hence contains minimal points with respect to the specialization order, which in our  topology on $X$ coincides with  set inclusion. Thus $Y \subseteq {\up(\min Y)}$. If in addition $Y$ is a  closed subspace of $X$, then $Y = \up(\min Y)$ since $Y \subseteq \up (\min Y) \subseteq Y$.  
Since $Z$ is a minimal $C$-representation, there exists a minimal  closed $C$-representation $Y$ of $A$ such that $Z = \min Y $. 
%Since $Y \subseteq \gen(\min(Y)) \subseteq \inv(Z) = Y$, it follows that $Y = \gen(Z)$. 
Since $Z \subseteq Y$ and $Y$ is  closed,  $\overline{Z} \subseteq Y$. Thus the minimality of $Y$ forces $\overline{Z} = Y$, and hence $\overline{Z}$ is a minimal closed $C$-representation of $A$.  Also,  note that this implies that $\min \overline{Z} = \min Y  = Z$.  

 Suppose that $Y$ is  a patch  $C$-representation of $A$ with $Y \subseteq \widetilde{Z}$.  We claim that $Y = \widetilde{Z}$, and we prove this by first showing that $Z \subseteq Y$.  By Proposition~\ref{Hoc prop}(3), $\up Y$ is a closed $C$-representation of $A$. Also by Proposition~\ref{Hoc prop}(3), $\up Y \subseteq \up (\widetilde{Z}) = \overline{Z}$, so that the minimality of $\overline{Z}$ forces $\up Y = \overline{Z}$. 
%By Proposition~\ref{Hoc prop}, $\overline{Z} \subseteq \up (\min \overline{Z})$, and 
Since $\overline{Z}$ is closed, $\up (\min \overline{Z}) = \overline{Z} = \up Y$. 
 We have established that $Z = \min \overline{Z}$, so $\up Y = \up Z$.
%Since $Y$ is a patch, Proposition~\ref{Hoc prop} implies  $Y$ is spectral. Thus  
 Now  $\min \up Y = \min Y$,  and since $Z$ consists of pairwise incomparable elements, $\min \up Z = Z$. Thus $\min Y = Z$, so that $Z \subseteq Y$.  
 Since $Y$ is a patch,  $\widetilde Z \subseteq Y$.  Therefore, $Y = \widetilde{Z}$.  
Since  $\min Y = Z$,   we conclude from $\widetilde{Z} = Y$ that
 $\min \widetilde{Z} = \min Y = Z$.  
 All that remains is to show that $\overline{Z} = \up Z$. Since $\min \overline{Z} = Z$, 
it follows that  $\overline{Z}  = \up \min \overline{Z} = \up Z$.  
\qed
\end{proof}

We obtain now a topological characterization of irredundance in minimal $C$-representations. Example~\ref{need to add} shows that without the restriction to minimal  representations, irredundance may not have a similar  topological expression. 

\begin{theorem} \label{surprised}   Let  $X$ be  a spectral $C$-representation of $A$, and let $Z$ be a minimal $C$-representation of $A$ in $X$.  
 Then the spectral and patch subspace topologies agree on $Z$, and the following are equivalent for $B \in Z.$ 
\begin{description}[(3)]
\item[{\em (1)}] $B$ is irredundant in $Z$. 

\item[{\em (2)}] 
$B$ is strongly irredundant in $Z$.
%\item[{\em (2)}] 

\item[{\em (3)}] 
$B$ is isolated in the spectral (equivalently, patch) subspace topology on $Z$

\end{description}
\end{theorem}

\begin{proof}   By assumption, there exists a minimal closed $C$-representation $Y$ of $A$ such that $Z = \min Y$.  
Since $Y$ is  a patch in $X$, we have by Proposition~\ref{Hoc prop}(4)(a) that $Y$ is spectral in the spectral subspace topology. Also by Proposition~\ref{Hoc prop}(4)(c), the spectral topology on the  minimal points of a spectral space is the same as  the patch topology. Thus the spectral and patch  topologies agree on $Z$. 

That (2) implies (1) is clear, and that (1) implies (3) follows from Lemma~\ref{ti}. 
% The converse  follows from 
%(1) $\Rightarrow$ (2)  By Lemma~\ref{remove}(1), $\overline{Z}$ is a minimal inverse closed representation of $C/A$, so by 
It remains to prove that (3) implies (2). 
 Suppose $B$ is  isolated in $Z$,
  so that  $B \not \in \overline{Z \setminus \{B\}}$.  To see that $B$ is strongly irredundant in $Z$,
  let $F$ be a closed subset of $\V(B)$ 
such that $  (Z \setminus \{B\}) \cup F$ is a $C$-representation of $A$.  Now $F \subseteq \V(B) \subseteq \overline{Z}$, so   
   $Z' := \overline{(Z \setminus \{B\})} \cup F$ is a closed $C$-representation of $A$ contained in $\overline{Z}$.  
    By  Lemma~\ref{remove},
  $\overline{Z}$ is a minimal  closed $C$-representation of $A$, so  $Z' = \overline{Z}$, and hence $B \in Z'$. Since $B \not \in \overline{Z \setminus \{B\}}$, we conclude that $B \in F$ and hence $F = \V(B)$.  Thus 
  %$B = \bigcap_{B' \in F}B'$, so that 
  $B$ is strongly irredundant in $Z$.
  %
%  such that $(Z \cup F) \setminus \{B\}$ is  a $C$-representation of $A$. 
% Now $F \subseteq \V(B) \subseteq \overline{Z}$, so   
%   $Z' := \overline{(Z \setminus \{B\})} \cup F$ is a closed $C$-representation of $A$ contained in $\overline{Z}$.  
%    By  Lemma~\ref{remove}(1),
%  $\overline{Z}$ is a minimal  closed $C$-representation of $A$, so  $Z' = \overline{Z}$, and hence $B \in Z'$. Since $B \not \in \overline{Z \setminus \{B\}}$, we conclude that $B \in F$ and hence $F = \V(B)$.  Therefore, $B$ is strongly irredundant in $Z$.  
\qed\end{proof}

\begin{corollary} \label{first cor} Let $X$ be a spectral $C$-representation of $A$, and let $Z$ be a minimal $C$-represent\-a\-tion  of $A$ in $X$. Then $Z$ 
 contains a (strongly) irredundant $C$-representation of $A$ if and only if   the set of isolated points in $Z$ is dense in  $Z$.
 Hence there is at most one irredundant $C$-representation of $A$ in $\widetilde{Z}$.  
\end{corollary}

\begin{proof}  By Theorem~\ref{surprised} the patch and spectral topologies agree on $Z$, so in the proof we work exclusively in the patch topology. 
Suppose $Z$ is a minimal $C$-representation of $A$ in $X$ that contains an irredundant $C$-representation $Y$ of $A$.
Since $Y$ is a $C$-representation of $A$, so is $\widetilde{Y}$. Thus since by Lemma~\ref{remove}, $\widetilde{Z}$ is a minimal patch $C$-representation of $A$, we have $\widetilde{Y} = \widetilde{Z}$. Therefore, $Y$ is dense in $Z$.  Moreover, by Lemma~\ref{ti}, each member of $Y$ is  irredundant in $\widetilde{Y} = \widetilde{Z}$, hence in $Z$.  Therefore, by Theorem~\ref{surprised}, each member of $Y$ is isolated in $Z$.

Conversely, suppose that the set $Y$ of isolated points in $Z$ is dense in $Z$.  
 If $Y$ is not a $C$-representation of $A$, then there exists $B' \in Z$ and  $c \in (\bigcap_{B \in Y}B) \cap C$ with $c \not \in B'$. Thus $Y \subseteq \V(c)$ and $B' \not \in \V(c)$, so that  $\U(c) \cap Y = \emptyset$ while $\U(c) \cap Z \ne \emptyset$, 
  a contradiction to the assumption that $Y$ is dense in $Z$.  Therefore, $Y$ is a $C$-representation of $A$, and hence by Theorem~\ref{surprised}, the members of $Y$ are strongly irredundant in the $C$-representation $Z$.    

To prove the last claim of the corollary, suppose $Y$ is an irredundant  $C$-represen\-ta\-tion of $A$ in $\widetilde{Z}$.  By Lemma~\ref{tfae},  the elements of $Y$ are isolated points in $\widetilde{Z}$ with respect to the patch topology. Thus for each $y \in Y$, $\{y\}$ is open in $\widetilde{Z}$, so that since $Z$ is  dense in $\widetilde{Z}$, we must have $y \in Z$. Therefore, $Y \subseteq Z$. Since $Y$ is an irredundant $C$-representation of $A$, 
Theorem~\ref{surprised} implies that 
$Y$ is the set of isolated points of $Z$. Hence there is at most one  irredundant $C$-representation of $A$ in $\widetilde{Z}$, namely the set of isolated points of $Z$.    
\qed\end{proof}

%\begin{remark} A similar argument shows that 
% If  $Z$ contains   a  tightly irredundant $C$-representation $Y $ of $A$, then $Y$ is the unique irredundant $C$-representation of $A$ in $\overline{R}$.   {\bf [does it?]} 
% \end{remark}

A topological space $X$ is {\it scattered} if every nonempty subspace $Y$ of $X$ contains a point  that is isolated in $Y$; equivalently, in every nonempty subset $Y$ of $X$ the set of isolated points in $Y$ is dense in $Y$. %2.11 in Guram's paper

\begin{corollary} \label{scattered cor} Let $X$ be a spectral $C$-representation of $A$. If $X$  is scattered in the spectral or patch topologies, then $X$ contains a strongly irredundant $C$-representation of $A$.  
\end{corollary}

\begin{proof} Since the patch topology refines the spectral topology, to be scattered in the spectral topology implies the space is  scattered in the patch topology. Thus we assume that $X$ is scattered in the patch topology. 
Let $Z$ be a minimal $C$-representation of $A$. Since $X$ is scattered, the set of isolated points in $Z$ is dense in $Z$ with respect to the patch topology, so  by Corollary~\ref{first cor}, 
$Z$ contains a strongly irredundant $C$-representation of $A$. 
\qed\end{proof}

\begin{corollary} \label{countable case}  If $X$  is a countable spectral $C$-representation of $A$, then $X$  contains a strongly  irredundant $C$-representation of $A$.  
\end{corollary}

\begin{proof}  
Since $X$ is spectral, the patch topology on $X$ is compact and Hausdorff. A countable compact Hausdorff space is homeomorphic to an ordinal space \cite{MS}, and hence scattered since  an ordinal space is well-ordered. 
Thus  
$X$ is scattered in the patch topology, and by 
Corollary~\ref{scattered cor}, $X$ contains  a strongly  irredundant $C$-representation of $A$
\qed
 \end{proof}

We single out next the members of $X$ that must appear in every closed $C$-representation of $A$. These members play an important role in the applications to intersections of valuation rings in Sections 5 and 6.

\begin{definition} Let $X$ be a spectral $C$-representation of $A$. Let $Y$ be the intersection of all closed $C$-representations of $A$  in $X$ (so that $Y$ is a closed set in $X$, but not necessarily a $C$-representation of $A$). An element $B \in X$ is {\it critical} for the $C$-representation $X$  if $B \in Y$. Since $Y$ is a closed subset of the spectral space $X$, $Y$ contains minimal elements. 
We define ${\cal C}(X) = \min Y$. 
%Thus the critical members of $X$ are precisely the elements of $X$ that contains a member of ${\cal C}(X)$.  

  %We say that $A$ is {\it critically $C$-represented} by $X$ if ${\cal C}(X)$ is a $C$-representation of $A$. 

 \end{definition}

 \begin{proposition} \label{vacant lemma}  Let $X$ be a spectral $C$-representation of $A$. 
 Then $B$ is critical in $X$ if and only if whenever $A = A_1 \cap \cdots \cap A_n \cap C$, where  each $A_i$  is an intersection of members of $X$, it must be that $A_i \subseteq B$ for some $i$. 

\end{proposition}

\begin{proof}
Suppose $B$ is critical in $X$ and   $A = A_1 \cap \cdots \cap A_n \cap C$, where  each $A_i$  is an intersection of members of $X$. Then $\V(A_1) \cup \cdots \cup \V(A_n)$ is a closed $C$-representation of $A$ in $X$. Since $B$ is critical in $X$,  $B \in \V(A_i)$ for some $i$, and hence $A_i \subseteq B$. 
Conversely,  suppose that whenever $A = A_1 \cap \cdots \cap A_n \cap C$, where  each $A_i$  is an intersection of members of $X$, we have $A_i \subseteq B$ for some $i$. Let $Y$ be a closed $C$-representation of $A$ in $X$.  Since $Y$ is closed and $X$ is spectral, $Y$ is an intersection of sets of the form $\V(F_1) \cup \cdots \cup \V(F_n)$, where each $F_i$ is a finite subset of $C$. Thus to show that $B \in Y$ it suffices to show $B$ is in every set of this form that contains $Y$. Let $F_1,\ldots,F_n$ be finite subsets of $C$ such that $Y \subseteq \V(F_1) \cup \cdots \cup \V(F_n)$.  For each $i$, let $A_i =   \bigcap_{E \in \V(F_i)}E$. Since $Y$ is a $C$-representation of $A$, we have  $A = A_1 \cap \cdots \cap A_n \cap C$.  By assumption, $A_i \subseteq B$ for some $i$, so  $F_i \subseteq B$.  Thus $B \in V(F_i)$, which proves the proposition.
   \qed
\end{proof}

The next corollary shows that for critical members of $X$, being irredundant in a representation is the same as being strongly irredundant.

\begin{corollary} \label{irredundant is si}  Let $X$ be a spectral $C$-representation of $A$. If  $B \in X$ is critical in $X$ and  $B$ is irredundant in some $C$-representation $Z$ of $A$ in $X$, then $B$ is strongly irredundant in $Z$.  
%If $F$ is a closed subset of $\V(B)$ such that $(Z \setminus \{B\})$ is a $C$-representation of $A$, then $F = \V(B)$.  
\end{corollary}

\begin{proof}
%By assumption, $A = B \cap (\bigcap_{B' \in Z \setminus \{B\}}B')$, with $B$ irredundant in this intersection.   
Suppose there is a nonempty  closed subset $Y$ of $\V(B)$ such that $(Z \setminus \{B\}) \cup Y$ is a $C$-representation of $A$. 
%Then $$A = (\bigcap_{B' \in F}B') \cap (\bigcap_{B' \in Z \setminus \{B\}}B') \cap C.$$
% By Proposition~\ref{vacant lemma}, $\bigcap_{B' \in F}B' \subseteq B$ or $\bigcap_{B' \in Z \setminus \{B\}}B'  \subseteq B$.  The latter case contradicts the irredundance of $B$ in $Z$, so  $B=  \bigcap_{B' \in F}B' $. 
We claim that  $B \in Y$.  Since $Y$ is closed, $Y$ is an intersection of sets of the form $\V(F_1) \cup \cdots \cup \V(F_n)$, where $F_1,\ldots,F_n$ are finite subsets of $ D$.  Let $F_1,\ldots,F_n$ be finite subsets of $ D$ such that  
$Y \subseteq \V(F_1) \cup \cdots \cup \V(F_n)$. 
 Then $$A=(\bigcap_{ B' \in \V(F_1)}B') \cap \cdots \cap (\bigcap_{B' \in \V(F_n)}B')  \cap (\bigcap_{B' \in Z \setminus \{B\}}B') \cap C  \subseteq B.$$  Since $B$ is irredunant in $Z$, we have $\bigcap_{B' \in Z \setminus \{B\}}B' \not \subseteq B$. Thus since $B$ is critical in $X$, we conclude that $F_i \subseteq \bigcap_{ B' \in \V(F_i)}B' \subseteq  B$ for some $i$. Hence $B \in  \V(F_1) \cup \cdots \cup \V(F_n)$, which shows that $B \in Y$. 
 % which shows $B$ is strongly irredundant in $Z$. 
\qed
\end{proof}

Next we prove a uniqueness theorem for strongly irredundant $C$-representations when ${\cal C}(X)$ has enough members to be itself a $C$-representation of $A$. This case is important in Section 6, where we work with  $v$-domains, a class of rings that can be represented as an intersection of their critical valuation overrings. 

\begin{theorem} \label{unique}  Let $X$ be a spectral $C$-representation of $A$. Then $A = (\bigcap_{B \in {{\cal C}(X)}}B) \cap C$  if and only if  $A$ has a unique minimal $C$-representation in $X$. If this is the case, then the following statements hold for  the set 
\begin{center}
$S = \{B \in X$: $B$ {\rm is strongly irredundant in some} $C${\rm-representation of} $A$ {\rm in} $X\}$.
\end{center}   
\begin{description}[(3)] 
\item[{\em (1)}]  $S \subseteq {\cal C}(X)$ and hence every member of $S$ is critical in $X$. 

\item[{\em (2)}] Each $B \in S$ is strongly irredundant in the $C$-represent\-a\-tion
 ${\cal C}(X)$.

\item[{\em (3)}]  
If $A$ has a strongly irredundant $C$-representation $Z$ in $X$, then $Z = S$.

\item[{\em (4)}]  $X$ contains at most one strongly irredundant $C$-representation of $A$.   

\end{description} 
\end{theorem} 

\begin{proof}
Observe first that $\up {\cal C}(X)$ is the intersection of all the closed $C$-representations of $A$ in $X$.  Thus 
 $A = (\bigcap_{B \in {{\cal C}(X)}}B) \cap C$ if and only if $\up {\cal C}(X)$ is a $C$-representation of $A$, if and only if there is a unique minimal closed $C$-representation of $A$, if and only if  there is a unique minimal $C$-representation of $A$ in $X$.   

(1)  Let $B \in S$. Then there is  $Y \subseteq X$ such that $Y \cup \{B\}$ is a $C$-representation of $A$ and  $B$ is strongly irredundant in $Y \cup \{B\}$. For $E \in {\cal C}(X)$,  
 Proposition~\ref{vacant lemma} implies $B \subseteq E$ or $\bigcap_{B' \in Y}B' \subseteq E$. If $\bigcap_{B' \in Y}B' \subseteq E$ for every $E \in {\cal C}(X)$, then since by assumption ${\cal C}(X)$ is a $C$-representation of $A$, this forces $A = (\bigcap_{B' \in Y}B') \cap C$, contrary to the irredundance of $B$ in  $\{B\} \cup Y$.  Therefore, $Z:= (\up {\cal C}(X)) \cap \V(B)$ is nonempty, and  for every $E \in {\cal C}(X)$ with $B \not \subseteq E$, it must be that $\bigcap_{B' \in Y}B' \subseteq E$.  Thus since ${\cal C}(X)$ is a $C$-representation of $A$, so is 
  $Z \cup Y$. Since $\up {\cal C}(X)$ and $\V(B)$ are closed subsets of $X$, so is $Z$. 
 Now $Z \subseteq \V(B)$, so  
   since $B$ is strongly irredundant in the $C$-representation $\{B\} \cup Y$, it must be that $Z = \V(B) $. Thus $B \in \up{\cal C}(X)$. 
 
 Next we show  that $B \in {\cal C}(X)$. There exists  $E \in {\cal C}(X)$ such that $E \subseteq B$. Since $A =   B \cap (\bigcap_{B' \in Y}B') \cap C$ and $E$ is critical, Proposition~\ref{vacant lemma} implies that $B = E$ or $\bigcap_{B' \in Y}B' \subseteq E$. Since $E \subseteq B$ and $B$ is irredundant in $\{B\} \cup Y$, the latter cannot occur, so  $B = E$. Therefore, $B$ is minimal in ${\cal C}(X)$.

 (2) As in (1), every $E \in {\cal C}(X)$ with $E \ne B$ contains $ \bigcap_{B' \in Y}B'$, so that if $B$ is not irredundant in ${\cal C}(X)$, then since ${\cal C}(X)$ is a $C$-representation of $A$, we have $A=(\bigcap_{B' \in Y}B')  \cap C$, a contradiction. Thus $B $ is irredundant in ${\cal C}(X)$, and by Corollary~\ref{irredundant is si}, $B$ is strongly irredundant in ${\cal C}(X)$.   
 
 (3) Suppose $Z$ is a strongly irredundant $C$-representation in $X$.  By (1), $Z \subseteq {\cal C}(X)$ and the members of $Z$ are strongly irredundant in ${\cal C}(X)$. Also by (2), the members of  $S$ are strongly irredundant in ${\cal C}(X)$. It follows that $S = Z$.  

(4) This is clear from (3). 
  \qed
\end{proof}

%\section{Noetherian subspaces of spectral spaces} 

%A topological space $X$ is  
 % {\it  scattered} if every nonempty  subspace  of $X$ contains a point that is  isolated in the subspace topology.  

%\begin{corollary} \label{Noetherian cor} Let $Y$ be a spectral space that is a tree under its  specialization order, and let $X$ be a Noetherian subspace of $Y$. Then every set of elements in $\gen(X)$ that are incomparable under the specialization order is discrete in the patch and inverse topologies. 
%\end{corollary} 

%\begin{proof} By {\bf [ref]}, $\gen(Y) = \inv(Z)$ and hence by {\bf [ref]}, $\gen(Y)$ is a spectral space.  Since $\Max(\gen(Y)) \subseteq Y$, $\Max(\gen(Y))$ is a Noetherian space, and hence Theorem~\ref{Noetherian prop} applies also to $\gen(Y)$.  
%\qed\end{proof} 

%\begin{corollary} Let $X$ be a spectral space. Every Noetherian subspace of $X$ that consists of elements that are incomparable under the specialization order is discrete in the inverse topology.
%\end{corollary} 

\section{Irredundance in intersections of valuation rings}

In this section, we reinterpret the material of Section 3 for the Zariski-Riemann space of a field. 
We assume the following notation throughout this section. 
\begin{description}[3]
\item[$\bullet$] 
$A$ is a proper integrally closed subring of a field $F$.

\item[$\bullet$] $C$ is a set (not necessarily a ring) such that $A \subsetneq C \subseteq F$. 

\item[$\bullet$]  $\X$ denotes the set of valuation rings of $F$ containing $A$. \end{description}

Zariski introduced a topology on $\X$ (the {\it Zariski topology}) by designating as a  basis of open sets the sets of the form $\{V \in \X:x_1,\ldots,x_n \in V\},$ where $x_1,\ldots,x_n \in F$.  With this topology, the same topology as in Example~\ref{spectral exs}(8), 
 $\X$ is a spectral space and is termed the {\it Zariski-Riemann space} of $F/A$. For some recent articles emphasizing a topological approach to  the Zariski-Riemann space, see \cite{Fin, FFL, FFL2, FFL3, FFS, FS,  ONoeth, ORej, OZR}.

Comparison of the basic opens in the Zarsiki topology on $\X$  with the  topology in Section 3 shows that it is the inverse topology on $\X$ rather than  the Zariski topology that is needed in order to deal with issues of irredundance. The inverse topology is also a  natural one to consider here since under this topology the specialization order on $\X$ agrees with the usual order on $\X$ given by set inclusion. To avoid confusing the two topologies, we denote by $\X^{-1}$ the set $\X$ with the inverse topology. 
Then $\X^{-1}$ is a spectral $C$-representation of $A$, and so all the results of Section 3 can be translated into the context of the Zariski-Riemann space of $F$ by working inside the spectral $C$-representation $\X^{-1}$. 

In the spirit of Section 3, we work throughout this section relative to the set $C$ and consider {\it $C$-representations $X$ of $A$}; that is,  $A = (\bigcap_{V \in X}V) \cap C$, where $X$ is a collection of valuation rings in $\X$. 
 That $C$ need only be a subset in most cases is a byproduct of the approach in  Section 2. 
While we do not have an application for the level of generality that working with a set rather than a ring affords, we do so anyway since it comes at no extra expense.  
When $C = F$, we abbreviate ``$C$-representation'' to ``representation.'' Thus a {\it representation} of $A$ is a subset $X$ of $\X$ such that $A = \bigcap_{V \in X}V$.

In this section $\X^{-1}$ will play the role that $X$ did in Section 3 of an ambient spectral representation.  In this section we use ``$X$'' then for  not necessarily spectral subsets of $\X$. A $C$-representation $X$ of $A$ then is a subspace of $\X^{-1}$. In particular:
\medskip
\begin{adjustwidth}{.5cm}{.5cm}{\it When applying the results of Section 3, the default topology on the $C$-representations of $A$ is the inverse topology. Thus the specialization order coincides with set inclusion among the valuation rings, and the operators ${\min}(-)$ and $\Max(-)$ yield the minimal and maximal elements, respectively, of a collection of valuation rings with respect to set inclusion.}
\end{adjustwidth}

%Since we will be shifting topologies throughout this section and the next, we use operator notation to denote the various closures we consider. Let $X \subseteq \X$. 

\begin{definition} 
{Let $X \subseteq \X$. We define $\cl(X),\inv(X)$ and $\patch(X)$ to be the closure of $X$ in the Zariski, inverse and patch topologies, respectively. We denote by $\gen(X)$ the set of generalizations of the valuation rings in $X$; that is, \begin{center}$\gen(X) = \{V \in \X:W \subseteq V$ for some $W \in X\}$.\end{center}}  
\end{definition}

%\begin{eqnarray*} 
%\cl(Z) & = & {\mbox{ Zariski closure of }} Z \\
%\inv(Z) & = & {\mbox{ closure of }} Z {\mbox{ in the inverse topology }} \\
%\end{eqnarray*} 

We interpret now the results of Section 3 in the setting of the Zariski-Riemann space. 
The notions of irredundance from Definition~\ref{long def}(2) can be simplified for valuation rings. Let $X $ be a subset of $  \X$ such that $A = (\bigcap_{V \in X}V) \cap C$. Then $V \in X$ is {\it irredundant} in the $C$-representation $X$ if $V$ cannot be omitted from this intersection; $V$ is {\it strongly irredundant} if $V$ cannot be replaced in this intersection by a valuation overring\footnote{By an {\it overring} of a domain $R$ we mean a ring between  $R$ and its quotient field.}; and $V$ is {\it tightly irredundant} if $V$ cannot be replaced by an intersection of  valuation overrings that properly contain $V$.

%\medskip
%{\noindent}{\textsc{Minimal representations.}} 
Combining Lemmas~\ref{minimal exist} and~\ref{remove}, and observing that in the notation of Section 2, $\gen(X) = \down X$, we have the following existence result for minimal representations.  
\medskip
\begin{adjustwidth}{.5cm}{.5cm}
{\bf (4.2)} {\it In every inverse closed subset  $X$ of $\X$ there is a  {\it minimal $C$-represent\-a\-tion} of $A$; that is, there exists in $X$ a collection $Z$ of pairwise incomparable valuation rings  such that $\gen(Z)$ is a minimal inverse closed $C$-representation of $A$ and $\patch(Z)$ is  a minimal patch  $C$-representation of $A$.}   
\end{adjustwidth}
\medskip
% and noting that sets of the form $\gen(Y)$, $Y \subseteq \X$, are inverse closed, 
 %this can be expressed without reference to topology:
%\medskip
%\begin{adjustwidth}{.5cm}{.5cm}
%{\bf (4.1)} {\it There exists a set $X$ of pairwise incomparable valuation rings in $\X$   such that $A = (\bigcap_{V \in X}V) \cap C$ and for each subset $Y$ of $\gen(X)$ such that $A = (\bigcap_{V \in Y}V) \cap C$,   it must be that $\gen(Y) = \gen(X)$.}  
%\end{adjustwidth}
%\medskip
% $\inv(X) =\gen(X)$ and $X$ contains no smaller subset with these same properties.  
In general, there can exist infinitely many such minimal $C$-representations of $A$.  This is illustrated by Example~\ref{three}. %We have then by Lemma~\ref{minimal exist} that {\it every inverse closed subset $X$ of $\X$ contains a minimal $C$-representation $Y$.}  (Since we are working always in the spectral $C$-representation $\X^{-1}$ of $A$, we drop the appellation ``in $\X^{-1}$'' from the reference to a minimal $C$-representation.) By  Lemma~\ref{}, $\inv(Y)$ is a minimal inverse closed $C$-representation of $A$, $\gen(Y) = \inv(Y)$ and $\patch(Y)$ is a minimal inverse closed $C$-patch representation of $A$. 

\setcounter{theorem}{2}

\begin{example} \label{three}  In \cite[Example 6.2]{OIrr}, an integrally closed overring $A$ of $K[X,Y,Z]$, with $K$ any field and $X,Y,Z$ indeterminates, is constructed such that $A$ has uncountably many strongly irredundant representations.
Since every valuation overring of $K[X,Y,Z]$ has finite Krull dimension, a valuation ring in a representation of $A$ is strongly irredundant if and only if it is tightly irredundant (see the discussion after (4.4)). 
 Thus by Proposition~\ref{count}, $A$ has uncountably many minimal representations. 
\end{example}

%\begin{proposition} Let $X \subseteq \X$, and suppose $X$ consists of valuation rings having finite Krull dimension. If $Z_1$ and $Z_2$ are distinct strongly irredundant $C$-representations of $A$, then $Z_1$ and $Z_2$ are contained in distinct minimal $C$-representations of $A$.  
%\end{proposition}  

%{\noindent}{\textsc{Irredundant representations.}}  

Applying Lemma~\ref{ti},  we have 
\medskip
\begin{adjustwidth}{.5cm}{.5cm}
%  \lipsum[1-2]
 {\bf (4.4)} {\it If $A = (\bigcap_{V \in X}V) \cap C$, then $V\in X$ is irredundant in $X$ if and only if $V$ is irredundant in $\patch(X)$; $V$ is tightly irredundant in $X$ if and only if $V$ is irredundant in $\inv(X)$. } 
\end{adjustwidth}
\medskip
 If $V$ has finite Krull dimension, then since there are only finitely many overrings of $V$,  $V$ is strongly irredundant in the $C$-representation $X$ if and only if $V$ is tightly irredundant in $X$.  More generally, if the maximal ideal of $V$ is not the union of the prime ideals properly contained in it, then the notions of strong and tight irredundance coincide for $V$. 
 In particular, if  $V \in X$ has rank one, then $V$ is irredundant in the $C$-representation $X$ if and only if $V$ is irredundant in $\inv(X)$. %(2) 
%In the terminology of the previous section, a valuation ring $V \in X$ is irredundant in $X$ if and only if $V \not \in \geom(X \setminus \{V\})$. Thus Theorem~\ref{TFAE irredundant} and Corollary~\ref{geom cor} yield characterizations of when $V$ is irredundant in $X$. 

%\medskip

%{\noindent}{\textsc{Irredundance in minimal representations.}} 
While an irredundant member $V$ of a $C$-representation $X$ is by Lemma~\ref{ti}
an isolated point in the inverse and patch topologies on $Z$, the converse need not be true, as illustrated by Example~\ref{need to add}. However, by restricting to minimal $C$-representations we obtain from Theorem~\ref{surprised} that irredundance is  topological  for such representations.
\medskip
\begin{adjustwidth}{.5cm}{.5cm}
{\bf (4.5)}  {\it Suppose  $X$ is a minimal $C$-representation of $A$, as in (4.2). 
A valuation ring $V \in X$ is irredundant in $X$ if and only if $V$ is strongly irredundant in $X$; if and only if $V$ is isolated in $X$ in the inverse (equivalently, patch) topology. }
\end{adjustwidth}
\medskip

\setcounter{theorem}{5}

\begin{example} \label{need to add} A valuation ring $V$ in a $C$-representation $X$ of $A$ may be isolated in the inverse topology on $X$ but be redundant in $X$. For example, let $A$ be an integrally closed Noetherian local  domain of Krull dimension $>1$, let $X = \{A_P: P$ is a height one prime ideal of $A\}$, and let $V$ be a DVR overring of $A$ that dominates $A$. Write the maximal ideal $M$ of $A$ as $M = (a_1,\ldots,a_n)$. Then $X$ is a subset of the inverse closed set $Y:= \{W \in \X:1/a_i \in W$ for some $i=1,\ldots,n\}$, while $V \not \in Y$. Thus $V$ is an isolated point in $\{V\} \cup X$ with respect to the inverse topology. However,  
since $A = \bigcap_{W \in X}W$, $V$ is redundant in the representation $\{V\} \cup X$ of $A$.  (The notion of a minimal representation remedies this: $X$ is a minimal representation so that $V$ is excluded from consideration since it is not an element of $X$.)   
\end{example}

%{\noindent}{\textsc{Strongly irredundant minimal representations.}}  
By Corollary~\ref{first cor}, the existence of a strongly irredundant $C$-representation within a minimal $C$-representation depends only on the topology of the minimal representation: 
\smallskip
\begin{adjustwidth}{.5cm}{.5cm}
{\bf (4.7)} $\:$ {\it Suppose  $X$ is a minimal $C$-representation of $A$, as in (4.2). 
Then  $X$ contains a strongly irredundant $C$-representation $Y$ of $A$  if and only if the set of isolated points in $X$ is dense in $X$ with respect to the inverse topology.} 
\end{adjustwidth}
\medskip
In such a case the only choice for $Y$ is the set of isolated points of $X$, and hence 
there exists at most one such irredundant $C$-representation of $A$ in $X$, hence also in  $\patch(X)$ (Lemma~\ref{ti}).  However, moving outside of $\patch(X)$, Example~\ref{three} shows there can exist infinitely many distinct strongly irredundant $C$-representations of $A$.  
This example involves an intersection of valuation overrings of a three-dimensional Noetherian domain. By contrast, strongly irredundant representations over two-dimensional Noetherian domains are much better behaved and have a number of uniqueness properties \cite{OIrr}.   

 \setcounter{theorem}{7}

%\smallskip

%{\noindent}{\textsc{Countable representations.}} 
One  consequence of the topological approach of Section 3  is an existence result for strongly irredundant $C$-represent\-a\-tions of $A$ in the countable case. This result, which follows from Corollary~\ref{countable case}, is revisited in the next section in Theorem~\ref{countable Prufer}. 
\medskip
\begin{adjustwidth}{.5cm}{.5cm}
 {\bf (4.8)} $\:$ {\it If $A = (\bigcap_{V\in X}V) \cap C$ for some countable patch $X$ in $\X$, then $X$ contains a strongly irredundant $C$-representation of $A$.}
   \end{adjustwidth}
\medskip
It is important here that we work with a countable patch rather than simply a countable subset of $\X$.  This is illustrated by the next example. 

\setcounter{theorem}{8}
\begin{example} 
 Suppose $A$ is a countable integrally closed  local Noetherian domain with maximal ideal $M$ and quotient field $F$. Suppose also that $A$ has Krull dimension $>1$. 
  Let $X$ be the collection of all DVR overrings $V$ of $A$ that are are centered in $A$ on $M$ and such that $V$ is a localization of the integral closure of some finitely generated $A$-subalgebra of $F$. Since $A$ is countable and Noetherian, there are countably many such valuation rings.  Moreover, $A$ is the intersection of the valuation rings in $X$, since if $x \in F \setminus A$, then there exists a valuation ring $V$ in $X$ whose maximal ideal contains $x^{-1}$, so that $x \not \in V$.  
%  
 % 
%  and every other valuation overring of $A$ is a patch limit point of valuation rings in $X$. This follows for example from the fact that every valuation overring of $A$ can be obtained as union of quadratic transforms of $A$ {\bf [ref]}.  Since the patch closure of $X$ is all of the valuation rings of $A$, $A$ is an intersection of valuation rings in $X$ {\bf [ref]}. 
If $V \in X$ is an irredundant  representative of $A$, then since the value group of $V$ is a subgroup of the group of rational numbers, $V$ is a localization of $A$ \cite[Lemma 1.3]{HONoeth}, a contradiction to the fact that $A$ has Krull dimension $>1$ and $V$ is centered on the maximal ideal of $A$.  Therefore, although $X$ is countable, $X$ contains no irredundant representatives of $A$. It follows from (4.8) that $X$ is not a patch closed subspace of $\X$. 
%
% Suppose $A$ is a countable two-dimensional regular local ring with maximal ideal $M$ and quotient field $F$. Let $X$ be the collection of all DVR overrings $V$ of $A$ that 
%are centered in $A$ on $M$ and such that $V$ is a localization of a finitely generated $A$-algebra. Since $A$ is countable, there are countably many such valuation rings, and every other valuation overring of $A$ is a patch limit point of valuation rings in $X$. This follows for example from the fact that every valuation overring of $A$ can be obtained as union of quadratic transforms of $A$ {\bf [ref]}.  Since the patch closure of $X$ is all of the valuation rings of $A$, $A$ is an intersection of valuation rings in $X$ {\bf [ref]}. However, if $V \in X$ is an irredundant  representative of $A$, then by {\bf [ref]}, $V$ is a localization of $A$, a contradiction to the fact that $A$ has Krull dimension $2$ and $V$ is centered on the maximal ideal of $A$.  Therefore, although $X$ is countable, $X$ contains no irredundant representatives of $A$. 
\end{example}

%\smallskip

%\begin{proposition} Let $A$ be an integrally closed ring between $D$ and $F$, and let $X$ be an inverse closed representation of $A$. If $P$ is a prime ideal of $A$, then $X$ contains an inverse closed representation of $A_P$. 
%\end{proposition} 

%\begin{proposition} Let $A$ be an integrally closed ring between $D$ and $F$, and let $X$ be an inverse closed representation of $A$ such that every valuation ring in $X$ other than $F$ is  a rational valuation ring centered on a maximal ideal of $A$. If $X$ is scattered, then $A$ is a Pr\"ufer domain. 
%\end{proposition} 

%Suppose that $X$ is a subspace of projective limit of a projective system ${\ff S}$ of  Noetherian spectral spaces. If there exists $X \in {\ff S}$ such that the given spectral map $f:X \rightarrow X$ has Noetherian fibers, then every closed point in $X$ 

%{\noindent}{\textsc{Critical valuation rings}.} 
Adapting the terminology from Section 3, we say 
 a valuation ring $V \in \X$ is {\it $C$-critical for $A$} if $V$ is an element of every inverse closed $C$-representation of $A$. Thus by Proposition~\ref{vacant lemma} and the fact that every integrally closed $A$-subalgebra of $F$ is an intersection of valuation rings in $\X$, we have 
\medskip
\begin{adjustwidth}{.5cm}{.5cm}
{\bf (4.10)} {\it $V$ is $C$-critical for $A$ if and only if whenever  $A_1,\ldots,A_n$ are integrally closed $A$-subalgebras of $F$ such that 
 $A = A_1 \cap \cdots \cap A_n \cap C$, it must be that $A_i \subseteq V$ for some $i$.}
 \end{adjustwidth}
\medskip
Also, from Corollary~\ref{irredundant is si} we see that if $V$ is $C$-critical for $A$  and irredundant in some $C$-representation $X$ of $A$, then $V$ is strongly irredundant in $X$. By restricting to the case where $C$ is an $A$-submodule of $F$, we obtain an important class of $C$-critical valuation rings; these are the valuation rings that play an important role in the next sections.  A valuation ring $V \in \X$ is {\it essential for $A$} if $V = A_P$ for a prime ideal $P$ of $A$. 

\setcounter{theorem}{10}

\begin{proposition}  \label{localization} Let $V \in \X$ such that $C \not \subseteq V$.  If $C$ is an $A$-submodule of $F$ and  $V$ is essential for $A$, then $V$ is  $C$-critical for $A$.
% \item[{\em (3)}] If  $A_1,\ldots,A_n$ are $A$-subalgebras of $F$, each of which is the integral closure in $F$ of a finitely generated $A$-subalgebra of $F$, and 
% $A = A_1 \cap \cdots \cap A_n \cap R$, then $A_i \subseteq V$ for some $i$.
\end{proposition}

\begin{proof}  We use Proposition~\ref{vacant lemma} to prove the claim. Let  $P$ be a prime ideal of $A$ such that $A_P = V$, and let $A_1,\ldots,A_n$ be integrally closed $A$-subalgebras of $F$ such that $A = A_1 \cap \cdots \cap A_n \cap C$. Then since localization commutes with finite intersections, we have $V = A_P = (A_1)_P \cap \cdots \cap (A_n)_P \cap C_P$. Since $V$ is  a valuation ring, the set of $V$-submodules  between $V$ and $F$ forms a chain. Therefore,  since $C \not \subseteq V$, there is $i$ such that $A_i \subseteq V$, and hence by Proposition~\ref{vacant lemma},  $V$ is $C$-critical.   
\qed \end{proof}

\begin{example} 
A  valuation overring that is $C$-critical for $A$ need not be essential. Suppose $A$ has quotient field $F$. Then $A$ is said to be {\it vacant} if it has a unique Kronecker function ring \cite{Fab}. (Kronecker function rings are discussed after 4.13.) As we see in  (4.14), this implies that $A$ has a unique minimal representation. Hence $A$ is vacant if and only if every valuation overring of $A$ is critical (see also \cite{Fab}).  As discussed in \cite{Fab} there exist vacant domains that are not Pr\"ufer domains, and hence  such a domain has a critical valuation overring that is not essential.  
\end{example} 
%\begin{corollary} The ring  $A$ is  vacant in $F$ if and only if every valuation overring of $A$ is critical. 
%\end{corollary}

%{\noindent}{\textsc{Irredundant critical valuation rings}}.  
Example~\ref{three}  shows that in general $A$ need not be an intersection of critical valuation overrings; equivalently, $A$ need not have a unique minimal $C$-representation. 
 However, for some  well-studied classes of rings, such as those in the next two sections, it is possible to represent $A$ with critical valuation rings. In this case, strong properties hold for $A$. For example, applying Theorem~\ref{unique}(1), we have the following fact. 
 \medskip
\begin{adjustwidth}{.5cm}{.5cm}
 {\bf (4.13)} {\it Suppose $A = (\bigcap_{V \in {\cal C}(\X)}V) \cap C$, where ${\cal C}(\X)$ is the set of minimal $C$-critical valuation rings in $\X$.    
  If $V \in \X$ is strongly  irredundant in some $C$-representation  of $A$, then $V$ is
in ${\cal C}(\X)$ and  $V$  is strongly irredundant in ${\cal C}(\X)$. }
 \end{adjustwidth}
\medskip
Thus ${\cal C}(\X)$ collects all the strongly irredundant representatives of $A$, and so, as in Theorem~\ref{unique}, having a strongly irredundant representation is a matter of having enough strongly irredundant representatives. 
\medskip
\begin{adjustwidth}{.5cm}{.5cm}
{\bf (4.13)} {\it Suppose $A =  (\bigcap_{V \in {\cal C}(\X)}V) \cap C$, so that ${\cal C}(X)$ is a $C$-representation of $A$. Then
$A$ has a strongly irredundant $C$-represent\-a\-tion  if and only if  $A$ is an intersection of $C$ with valuation rings in the  set 
\smallskip
\begin{center}
$\{V \in \X: V$ {is  strongly irredundant in some} $C$-{representation of} $A\}$.
\end{center} \smallskip
Thus  
$A$ has at most one strongly irredundant $C$-represen\-ta\-tion.}
 \end{adjustwidth}
\medskip

%{\noindent}{\textsc{Kronecker function rings}}. 
There is a long tradition of using  Kronecker function rings to represent integrally closed rings in the field $F$ with a B\'ezout domain in a transcendental extension $F(T)$ of $F$. We depart from this tradition because of our emphasis on the more general topological approach via spectral representations as in Section 3. However, in the present context of Zariski-Riemann spaces there is a precise connection between  minimal representations and maximal Kronecker function rings. In fact,  
 minimal representations play for us  a role similar to that played by the Kronecker function ring in articles such as \cite{Bre, BM, GHOlber, ONoeth}. We outline this connection here.

 Let  $T$ be an indeterminate for $F$.  For each valuation ring $V \in {\ff X}$, let  $V^*  = V[T]_{{\ff M_V}[T]},$ where ${\ff M}_V$ is the maximal ideal of $V$.  Then $V^*$ is a valuation ring with quotient field $F(T)$ such that $V = V^* \cap F$. For a nonempty subset ${{X}}$ of ${\ff X}$,  the {\it Kronecker function ring of ${{X}}$} is the ring $$\Kr({{X}}) = \bigcap_{V \in {{X}}} V^*.$$  
Then 
 $\Kr(X)$ is a B\'ezout domain with quotient field $F(T)$; cf.~\cite[Corollary 3.6]{FFL}, \cite[Theorem 2.2]{HalKoc} and   
  \cite[Corollary 2.2]{HK}. When $X$ is a $C$-representation of $A$, then $A = \Kr(X) \cap C$, and we say that $\Kr(X)$ is a {\it Kronecker $C$-function ring of $A$}.  
Thus to every $C$-representation of $A$ corresponds a Kronecker $C$-function ring of $A$. 

For each $X \subseteq \X$, let $X^* = \{V^*:V \in X\}$. The mapping $\X \rightarrow \X^*$ is a homeomorphism with respect to the Zariski topology (see \cite[Corollary 3.6]{FFL} or \cite[Proposition 2.7]{HK}), and hence is a homeomorphism in the inverse and patch topologies also.  The subset $X$ is inverse closed in $\X$ if and only if $X^*$ is the set of localizations at prime ideals  of $\Kr(X)$; i.e., $X^*$ is the Zariski-Riemann space of the B\'ezout domain $\Kr(X)$ \cite[Proposition 5.6]{OZR}.  Moreover, we have the following connection between $C$-representations and Kronecker $C$-function rings, which can be deduced from \cite[Corollary 5.8 and Proposition 5.10]{OZR}. 

\medskip
\begin{adjustwidth}{.5cm}{.5cm}
{\bf (4.14)} {\it 
 The inverse closed $C$-representations of $A$ bijectively  correspond  to the Kronecker $C$-function rings of $A$. The minimal $C$-representations of $A$ bijectively  correspond to the maximal Kronecker $C$-function rings of $A$. Moreover, a subset $X$ of $\X$ is a minimal $C$-representation of $A$ if and only if $X^*$ consists of the  localizations at maximal ideals of a maximal Kronecker $C$-function ring of $A$.}
   \end{adjustwidth}

%\setcounter{theorem}{14}

%\begin{proposition} \label{KFR} A nonempty subset $X$ of $\X$ 
%\begin{description}[(3)] 
%
%\item[{\em (1)}] The set $X$   is a minimal inverse closed $C$-represent\-a\-tion of $A$ if and only if $X^*$ consists of the valuation overrings of a maximal Kronecker $C$-function ring of $A$. 
%
%\item[{\em (2)}] The set $X$ 
%is a minimal $C$-representation of $A$ if and only if $X^*$ consists  of the localizations at maximal ideals of a maximal Kronecker $C$-function ring of $A$.%\footnote{Could use ``iff $\delta(Z) = \Max(\Kr(A))$.''}  
%\end{description}
%\end{proposition} 

%\begin{proof} 
%(1) Suppose $X$ is a minimal inverse closed $C$-representation of $A$. 
% Suppose $X$ is a minimal $C$-representation of $A$. By {\bf [ref]}, $X^*$ is the set of localizations of the Kroencker $C$-function ring $\Kr(X)$ at its maximal ideals. If there exists a larger Kronecker $C$-function ring, then its valuation overrings contract to a set of valuation rings in $\X$ that is an inverse closed $C$-representation of $A$ properly contained in $\inv(X)$ {\bf [ref]}, a contradiction to the minimality of $X$. Thus $\Kr(X)$ is a maximal Kronecker $C$-function ring.   

%Conversely, if $X^*$ is the set of localizations of a maximal Kronecker $C$-function ring of $A$, then $X = \min(Y)$ for the inverse closed subset $Y$ of $\X$ such that $Y^*$ is the set of valuation overrings of $\Kr(X)$. Since $\Kr(X)$ is a maximal Kroencker $C$-function ring, $Y$ is by {\bf [ref]}  a maximal inverse closed $C$-representation. 
%\qed\end{proof}

\section{Generalizations of Krull domains}

In this section we assume the same notation as Section 4. Thus $A$ is an integrally closed subring of the field $F$, $C$ is a set between $A$ and $F$, and $\X$ is the Zariski-Riemann space of $F/A$.  
Intersection representations  play an important role in the theory of {\it Krull domains}, those integral domains that can be represented by a finite character intersection of rank one discrete valuation rings (DVRs). (A subset $X$ of $\X$ has {\it finite character} if each $0 \ne x \in F$ is a unit in all but at most finitely many valuation rings in $X$.)  Finite character representations of a Krull domain $A$ are well understood: The collection $X = \{A_P:$ $P$ a height one prime ideal of $A\}$ is a finite character, irredundant representation of DVRs. % and any other finite character representation of $A$ by rank one valuation rings must contain $X$.  
Krull \cite{Kru} proved more generally that if $A$ has a finite character representation consisting of valuation rings whose value groups have rational rank one, then this collection can be refined to one in which every valuation ring is essential for $A$; see also \cite[Corollary 5.2]{OOlber}. Examples due to Griffin \cite[Section 4]{Gri2}, Heinzer and Ohm \cite[2.4]{HOOlb} and  Ohm \cite[Example 5.3]{OOlber}  show that the same is not true if the value groups of the valuation rings are assumed only to  have rank one rather than rational rank one.  

%A number of authors have generalized various parts of the theory of Krull domains to other class of rings having a finite character representation. For example, 
Griffin defines the ring $A$ to have {\it Krull type} if $A$ has a finite character representation $X$ consisting of essential valuation rings  \cite{Gri2, Gri}. In \cite[Theorem 7]{Gri2} he gives necessary conditions for a ring $A$ having a finite character representation of valuation rings to be a ring of Krull type. Pirtle \cite[Corollary 2.5]{Pir} showed that when in addition the valuations in $X$ have rank one, $X$ is an irredundant representation of $A$.  
More generally, 
Brewer and Mott  \cite[Theorem 14]{BM} prove that if $A$ has a finite character representation $X$ of valuation rings (no restriction on rank), then $A$ has an irredundant finite character representation, and if also the valuation rings in $X$ have rank one,  %a case first studied by Ribenboim \cite{Rib},
 then $A$ has one, and only one, irredundant finite character representation consisting of rank one valuation rings \cite{BM}. In \cite[Theorem 1.1]{Bre}, Brewer proves that if $A$ has Krull type, then $A$ has an irredundant finite character representation $X$ consisting of essential valuation rings, and that $X$ is unique among such representations.  
 
 In both the articles \cite{Bre} and \cite{BM}, the authors prove their results by passing to a maximal Kronecker function ring of $A$ and applying Gilmer and Heinzer's theory of irredundant representations of Pr\"ufer domains  to work out the problem of irredundance in a Pr\"ufer setting. 
This method  of passage to a Kronecker function ring, and hence to a Pr\"ufer domain, is applied in \cite{ORej} to domains $A$ that can be represented with a collection of valuation rings from  a Noetherian subspace of the Zariski-Riemann space, a class of representations that subsumes the finite character ones. 
 In such a case $A$ can be represented by a strongly irredundant Noetherian space of valuation rings \cite[Theorem 4.3]{ORej}.
 The results in \cite{ORej} are in fact framed in terms of $C$-representations, where $C$ is a ring.\footnote{The results in \cite{ORej} also apply to representations consisting of integrally closed rings, not just valuation rings. In light of  Finocchiaro's theorem that the space of integrally closed subrings of $F$ is a spectral space (see Example~\ref{spectral exs}(6)), 
 it seems likely that  this level of generality might be handled with spectral $C$-representations also.}

The introduction of finite character rank one $C$-representations to generalize the theory of Krull domains is due to Heinzer and Ohm  \cite{HOOlb}. This allows for considerable more flexibility in applying results to settings in which one considers, say, integrally closed rings between $A$ and some  integrally closed overring $C$. Even when $A$ is a two-dimensional Noetherian domain and $C$ is chosen a PID, the analysis of the integrally closed rings between $A$ and $C$   is quite subtle; see for example  \cite{AHE, CP, LTOlber, ONoeth, OT}.    Regardless of the choice of $A$ and $C$, Heinzer and Ohm \cite[Corollary 1.4]{HOOlb} prove that finite character rank one $C$-representations   remain as well behaved as in the classical case of $C=F$: If $C$ is a ring and $A$ has a $C$-representation consisting of rank one valuation rings, then $A$ has a unique  irredundant finite character representation consisting of rank one valuation rings.  
%We discuss the approach of Heinzer and Ohm in more detail in Remark~\ref{HO remark}. 

In this section we   recover the above results using the topological methods developed in Section 3 and elaborated on in Section 4. %We are able to do so with the exception of a special case of a result of Heinzer and Ohm, an exception  discussed in Remark~\ref{HO remark}. 
Whereas in the articles \cite{Bre, BM, ORej} the strategy is to pass to a maximal Kronecker function ring and treat irredundance there, we work instead with the topology of minimal representations to obtain our results.    We also need only that $C$ is a set. 
A $C$-representation $X$ is {\it Noetherian} if $X$ is a Noetherian subspace of $\X$ with respect to the Zariski topology.

\begin{theorem} \label{main Noetherian} If $A = (\bigcap_{V \in X}V)\cap C$, where $X$ is a Noetherian subspace of $\X$ , then $\gen(X)$ contains a Noetherian  strongly irredundant $C$-representation of $A$.   
 \end{theorem} 

\begin{proof}
Since $X$ is Noetherian, $X$ is quasicompact, and hence $\inv(X) =\gen(X)$ \cite[Proposition 2.2]{OZR}. %Thus by Proposition~\ref{Hoc prop}, $\gen(X)$ is a spectral space under the Zariski topology. %The elements of min$(\gen(X))$ are the maximal elements of $\gen(X)$ under the specialization order of the Zariski topology. 
 Thus by (4.2), there exists a minimal $C$-representation $Y$ of $A$ in $\gen(X)$.  By Proposition~\ref{Hoc prop}(1), $\gen(X)$ is a spectral space under the Zariski topology. The elements of $\min X$ are the maximal elements of $\gen(X)$ under the specialization order of the Zariski topology. 
In particular, $\min X \subseteq X$.   
 Thus since $\min X$ is Noetherian in the Zariski topology, 
 Theorem~\ref{Noetherian prop} implies that  $Y$ is discrete in the  inverse topology, so that each valuation ring in $Y$ is an isolated point in $Y$ in the inverse topology. By (4.5), the valuation rings in the minimal $C$-representation $Y$ are strongly irredundant. Also, by Lemma~\ref{technical Noetherian}, $Y$ is  Noetherian in the Zariski topology. 
\qed\end{proof}

\begin{remark} 
 In general, the strongly irredundant $C$-representation in Theorem~\ref{main Noetherian} is not unique. For example, the uncountably many strongly irredundant representations of the  ring $A$ discussed in Example~\ref{three} are Noetherian spaces in the Zariski topology.  The ring $A$ in this case is an overring of a three-dimensional Noetherian domain. By contrast, when $C$ is integrally closed and $A$ is an overring of a two-dimensional Noetherian domain, with $A$ representable by a  Noetherian space of valuation overrings, then $A$ has a unique strongly irredundant $C$-representation \cite[Corollary 5.7]{OIrr}. 
In the case in which $A$ is an overring of a two-dimensional Noetherian domain, the existence of a Noetherian $C$-representation has strong implications for the structure of $A$; see \cite{ONoeth}.
\end{remark}

%From the theorem we derive a corollary that improves on several results in the literature. A collection of valuation rings in $\X$ has {\it finite character} if each nonzero $x \in F$ is a unit in all but at most finitely many valuation rings in $\X$.  {\bf Discuss; lead into introduction of $C$ and strong irredundance}.  These were all proved by passing to a maximal Kronecker function ring for $A$. In our case it follows from our more elementary  approach  via the results of Section 2. 

From the theorem we deduce a corollary that recovers a number of the results discussed at the beginning of the section, with the additional feature that the valuation rings in the representation are strongly irredundant rather than just irredundant. When the valuations are essential, we also obtain uniqueness across all strongly irredundant representations, not just the finite character ones.   
%This distinction is important for the uniqueness results  in Theorem~\ref{v theorem}. 

\begin{corollary} \label{FC cor}  If $A = (\bigcap_{V \in X}V)\cap C$, where $X$ is a finite character subset of $\X$, then $\gen(X)$ contains a  strongly irredundant $C$-representation $Y$ of $A$. If also each valuation ring in  $X$ is essential for $A$, then $Y$ is the only strongly irredundant $C$-representation of $A$ in $\X$.  
\end{corollary} 

\begin{proof}
A finite character collection of valuation rings in $\X$ is Noetherian in the Zariski topology \cite[Proposition 3.2]{ORej}, so by Theorem~\ref{main Noetherian}, $\gen(X)$ contains  a strongly irredundant  $C$-represent\-a\-tion $Y$. %Also,  since $Y \subseteq \gen(X)$ and $X$ has finite character, it follows that $Y$ has finite character.  
If also  every valuation ring in $X$ is essential for $A$, then every valuation ring in $\X$ that does not contain $C$ is $C$-critical for $A$ (Proposition~\ref{localization}), so the assertion of uniqueness  follows from (4.13). 
\qed\end{proof}

If $A$ has quotient field $F$ and $A = V \cap R$, where $V$ is a rational valuation overring of $A$ (i.e., $V$ has  value group isomorphic to a  subgroup of the rational numbers) and $R$ is a subring of $F$ properly containing $A$, then $V = A_{P}$, where $P$ is the prime ideal of $A$ that is contracted from the maximal ideal of $V$ \cite[Lemma 1.3]{HONoeth}. 
Applying this to the setting of Corollary~\ref{FC cor}, we recover the result of Krull discussed at the beginning of the section, but strengthened to guarantee uniqueness across all strongly irredundant representations. 

\begin{corollary} \label{finite cor3} Suppose $C$ is a ring and  $A = (\bigcap_{V \in X}V)\cap C$, where $X$ is a finite character  representation of $A$ consisting of valuation rings of rational rank one. 
Then $X$ contains a strongly irredundant  $C$-representation $Y$ of $A$, and $Y$  is the only strongly irredundant $C$-representation of $A$.
\end{corollary}

\begin{remark} \label{HO remark}
Heinzer and Ohm prove a version of
 Corollary~\ref{finite cor3} for finite character $C$-representations $X$ of $A$ when $A$ consists of rank one valuation  rings (in their terminology, {\it $A$ is a $C$-domain of finite real character}). Thus their approach includes rank one valuation rings with irrational value group also. Unlike rational valuation rings, such valuation rings can be strongly irredundant in $X$ but not essential for $A$; see \cite[Section 2]{HOOlb}. They prove that if $A$ has quotient field $F$, $C$ is a ring and  $A = (\bigcap_{V \in X}V)\cap C$, where $X$ is a finite character subset of $\X$ consisting of rank one valuation rings, then any valuation ring that is irredundant in some $C$-representation of $A$ is a member of every finite character $C$-representation of $A$ that consists of rank one valuation rings, and the collection of all such valuation rings is a $C$-representation of $A$ \cite[Corollary 1.4]{HOOlb}.   It seems plausible that our approach can recover this result  also, but more information is needed about $C$-representations.  
 \end{remark}

%\begin{proof} By Corollary~\ref{finite cor2}, $A = \bigcap_{V \in X}V$, where $X$ is the collection of valuation rings that are localizations of $A$ at height one prime ideals. 

%\qed\end{proof}

  \section{Pr\"ufer ${\pmb{v}}$-multiplication domains}

 We assume throughout this section that $A$ is an integrally closed domain a quotient field $F$, and that $\X$ is the Zariski-Riemann space of $F/A$. We no longer work with an intermediate set $C$, or more precisely, we work with $C = F$. Hence we drop $C$ from our usual terms such as ``$C$-representation'' and ``$C$-critical'' and simply write ``representation'' for ``$F$-representation'' and ``critical'' for ``$F$-critical.''

 The concept  of  a Pr\"ufer $v$-multiplication domain encompasses that of a  Krull domain and a Pr\"ufer domain, as well as  polynomial rings over these domains.  In this section we apply the point of view developed in Section 4 to the issue of irredundance in representations of Pr\"ufer $v$-multiplication domains.  While the results in this section shed additional light on the domains of  Krull type considered in the last section, the real impetus for the section comes from the theory of Pr\"ufer domains. We use the topological methods of Sections 3 and 4 to recover irredundance results for this class of rings as well as generalize them 
to Pr\"ufer $v$-multiplication domains.

For an ideal  $I$  of  $A$, let $I_v = (A:_F (A:_F :I))$.  
An integral domain $A$ is a {\it $v$-domain} if whenever $I,J,K$ are ideals of $A$ such that $(IK)_v = (JK)_v$, then $I_v = J_v$.  Examples of such domains include completely integrally closed domains and Pr\"ufer $v$-multiplication domains; for a recent survey of  this class of rings, see \cite{FZ}.  A $v$-domain $A$ has a unique maximal Kronecker function ring \cite[Theorem 28.1]{Gil}, so by (4.14), $A$ has a unique minimal representation. In particular, $A$ is an intersection of its critical valuation rings.

\begin{theorem} \label{v theorem} A $v$-domain  has at most one strongly irredundant representation. 
%Moreover, $A$ can be represented by a strongly irredundant intersection of valuation rings in $\X$ if and only if $A$ is the intersection of all the valuation rings in $\X$ that are strongly irredundant in some representation of $A$ as an intersection of valuation rings. 

\end{theorem}  

\begin{proof}  Since a $v$-domain $A$ is an intersection of critical valuation overrings, we may apply (4.14) in the case where $C = F$ to obtain the theorem. 
\qed\end{proof}

\begin{remark} In \cite[p.~310]{GHOlber}, Gilmer and Heinzer ask whether it is the case that if $A$ is a $v$-domain that is an irredundant intersection of valuation rings,  then the unique maximal Kronecker function ring of $A$ is an irredundant intersection of valuation rings. We can answer this question in the affirmative under the stipulation that $A$ is represented as a {\it strongly} irredundant intersection of valuation overrings.  In this case, by (4.13), each strongly irredundant representative of $A$ is contained in the minimal representation ${\cal C}$ of $A$ consisting of the minimal critical valuation rings for $A$.  By (4.14),  $\Kr({\cal C})$ is the unique maximal Kronecker function ring of $A$.  Since ${\cal C}$ is a minimal representation of $A$, (4.7) implies ${\cal C}$ contains a dense set of isolated points in the inverse topology. Thus ${\cal C}^*$ has a dense set of isolated points, so that by (4.7), $\Kr({\cal C})$ has a strongly irredundant representation. 
\end{remark}

For each ideal $I$ of $A$, let $I_t = \sum_J J_v$, where $J$ ranges over the finitely generated ideals of $A$ contained in $I$. An ideal $I$ of $A$ is a {\it $t$-ideal} if $I = I_t$. If $I$ is maximal among $t$-ideals, then $I$ is a {\it $t$-maximal ideal}. A $t$-ideal $I$ is a {\it $t$-prime ideal} if $I$ is prime. A $t$-maximal ideal is a prime, hence $t$-prime, ideal.  The set of $t$-prime ideals is denoted  $t${\rm-Spec}\,$A$, while the set of $t$-maximal ideals is denoted $t${\rm-Max}\,$A$.
    The domain $A$ is a {\it Pr\"ufer $v$-multiplication domain (P$v$MD)} if  every nonzero finitely generated ideal  of $A$
 is $t$-invertible; equivalently, $A_M$ is a valuation domain for each $M \in    t${\rm-Max}\,$A$ \cite[Theorem 3.2]{Kang}.  Every P$v$MD $A$ is an {\it essential domain}, meaning that $A $ is an intersection of essential valuation overrings. %  , $(II^{-1})_v = A$ and $I^{-1}$ has finite type.  

In a recent article, Finocchiaro and Tartarone \cite[Corollary 2.6]{FT} show that an essential domain is a P$v$MD if and only if the set of  its essential valuation overrings   is patch closed.  We use this characterization to interpret P$v$MDs in terms of critical valuation rings. 

\begin{lemma} \label{new PVMD} 
%Suppose $A$ has quotient field $F$. Then the following are equivalent. 
%\begin{description}[(3)]
%\item[{\em (1)}] 
The domain $A$ is a P$v$MD if and only if 
%\item[{\em (2)}] $A$ is a $v$-domain for which every critical valuation for $A$ is a localization of $A$.  
%\item[{\em (3)}] 
$A$ is an essential domain for which every critical valuation ring of $A$ is essential. 
%\end{description}
\end{lemma}

\begin{proof}
%(1) $\Rightarrow$ (2) 
Since $A$ is a P$v$MD, $A$ is an essential  domain. By \cite[Corollary 2.6]{FT}, the set $E$ of essential valuation rings  is patch closed in the Zariski-Riemann space of $A$.  If $V \in E$, then so is every overring of $V$, so  $\gen(E) = E$. Thus Proposition~\ref{Hoc prop}(3) implies that $E$ is an inverse closed representation of $A$, and hence  every critical valuation ring of $A$ is in $E$. 

%(2) $\Rightarrow$ (3) Since $A$ is a $v$-domain, $A$ is an intersection of its critical valuation rings. Thus by (2), $A$ is an essential domain. 
Conversely, suppose every  critical valuation ring of $A$ is essential.  Then by Proposition~\ref{localization},  the set $E$ of essential valuation rings is equal to the set of critical valuation overrings of $A$.  Therefore,  $E$ is inverse closed, hence patch closed, so that by \cite[Corollary 2.6]{FT}, $A$ is  a P$v$MD.  
\qed\end{proof}

%In the next theorem and its proof, by a {\it representation} of $A$ we mean a collection of valuation overrings of $A$ that intersect to give $A$. 

\begin{theorem} \label{Prufer} Suppose $A$ is a P$v$MD. Let $V$ be a valuation overring of $A$, and let $P$ be the center of $V$ in $A$.  Then the following are equivalent.
\begin{description}[(3)]

\item[{\em (1)}]   $V$ is strongly irredundant  in some representation of $A$. 

\item[{\em (2)}] $V = A_P$, $P \in t${\rm-Max}\,$A$ and $P$ is  isolated in $t${\rm-Max}\,$A$  in the inverse topology.

\item[{\em (3)}]   $V = A_P$, $P \in t${\rm-Max}\,$A$ and $P$ contains  
 a finitely generated ideal that is not contained in any other $t$-maximal ideal.

  \end{description}
  \end{theorem}

  \begin{proof} (1) $\Rightarrow$ (2) Suppose $V$ is strongly irredundant in some representation  of $A$. Since $A$ is a P$v$MD, 
 $A$ is an essential domain, and hence $A$ is an intersection of its critical valuation rings. Therefore, by (4.13), $V$ is  in the collection ${\cal C}$ of the valuation rings that are minimal among critical valuation rings for $A$.
 By Lemma~\ref{new PVMD}, the set of 
critical valuation rings for $A$ 
  coincides with the set of essential valuation overrings   of $A$.   Thus $V = A_P$, and 
 since $V \in {\cal C}$, $P  \in t${\rm-Max}\,$A$.    By (4.13), $V$ is strongly irredundant in ${\cal C}$, so that by (4.5), $V$ is isolated in ${\cal C}$ with respect to the inverse topology.  
Since ${\cal C} = \{A_Q:Q \in t${\rm-Max}\,$A\}$, the map  that sends a valuation overring of $A$ to its center in $A$ restricts to a  homeomorphism from ${\cal C} $ onto $t${\rm-Max}\,$A$. (That this map is continuous and closed follows from \cite[Lemma 5, p.~119]{ZS}.)
%Since $A$ is a Pr\"ufer domain, the canonical map $\X \rightarrow \Spec A$ that sends a valuation ring in $\X$ to its center in $A$ is a homeomorphism.
 Thus 
 $P$ is a $t$-maximal ideal of $A$ that is isolated in the inverse topology of $t${\rm-Max}\,$A$.  
 
 (2) $\Rightarrow$ (3) By (2), 
 there is a Zariski quasicompact open subset of $t${\rm-Max}\,$A$ whose complement in $t${\rm-Max}\,$A$ is $\{P\}$. Statement (3) now follows.%\footnote{$\U(I)  \cap t\Max(R) = t\Max(R) \cap \bigcup_\alpha \U(I_\alpha)$. $\U(I) \cap t\Max(R)$ quasicompact implies $\U(I) \cap t\Max(R) = \U(I_{\alpha_1} + \cdots + I_{\alpha_n})$. } 
 
 (3) $\Rightarrow$ (2) This is clear.

  (2) $\Rightarrow$ (1) As in the proof that (1) implies (2), the canonical map  ${\cal C} \rightarrow t${\rm-Max}\,$A$ is a homeomorphism, so 
  we conclude that $V$ is isolated in the inverse topology on ${\cal C}$.  Since ${\cal C}$ is a minimal representation of $A$,  (1) follows from (4.5). 
  \qed\end{proof}

\begin{remark} \label{GH remark}   If $V$ is a valuation overring of the domain $A$ that is irredundant in some representation of $A$ as an intersection of valuation overrings and $V$ is essential, then, as discussed after (4.10), $V$ is strongly irredundant. Thus when $A$ is a Pr\"ufer domain, $V$ is irredundant in a representation of $A$ if and only if it is strongly irredundant. Therefore, when $A$ is a Pr\"ufer domain, ``strongly irredundant'' can be replaced by ``irredundant'' in Theorem~\ref{Prufer}(1). With this replacement, the theorem recovers a  characterization of irredundant representatives of Pr\"ufer domains due to Gilmer and Heinzer; cf.
 \cite[Proposition 1.4 and Lemma 1.6]{GHOlber}.   
\end{remark} 
  
  In light of Remark~\ref{GH remark}, the next corollary is a version of  \cite[Theorem 1.10]{GHOlber} for 
 P$v$MDs. 
  
  \begin{corollary}
Suppose $A$ is a  P$v$MD. Then $A$ can be represented (uniquely) as a strongly irredundant intersection of valuation overrings  if and only if there is a collection $X$ of $t$-maximal ideals of $A$ such that
\begin{description}[(3)] 
\item[{\em (a)}]  each nonzero element of $A$ is contained in  
 a member of $X$, and
\item[{\em (b)}] each $P \in X$ contains  a finitely generated ideal that is not contained in  any other $t$-maximal ideal of $A$. 
\end{description}
\end{corollary}

\begin{proof} Suppose that $A$ has  a strongly irredundant representation $Z$.   By Theorem~\ref{v theorem} this representation is unique, and by Theorem~\ref{Prufer} there is a subset $X$ of $t${\rm-Max}\,$A$ such that $Z = \{A_P:P \in X\}$.  Since $A = \bigcap_{P \in X}A_P$, every nonzero element of $A$ is contained in a member of $X$.  Moreover, by Theorem~\ref{Prufer}, each $t$-maximal ideal in $X$ contains a finitely generated ideal that is not contained in any other $t$-maximal ideal. 
Conversely, if $X \subseteq t${\rm-Max}\,$A$ such that (a) and (b) hold for $X$, then by (a), $Z:=\{A_P:P \in X\}$ is a representation of $A$, and by Theorem~\ref{Prufer}, each $V \in Z$ is strongly irredundant in some representation of $A$. Therefore, by Theorem~\ref{v theorem}, $A$ has a strongly irredundant representation. 
\qed\end{proof}

\begin{remark}  The ring of integer-valued polynomials Int$({\mathbb{Z}})$ is the set of all polynomials $f(X) \in {\mathbb{Q}}[X]$ such that $f({\mathbb{Z}}) \subseteq {\mathbb{Z}}$.
Among the many well-known properties of this interesting  ring is that it is a Pr\"ufer domain; see  \cite{CC}.   
In the recent article \cite{CP}, Chabert and Peruginelli characterize all the rings $R$ between Int$({\mathbb{Z}})$ and ${\mathbb{Q}}[X]$ that can be represented as an irredundant intersection of valuation overrings. These are precisely the intermediate rings $R$ such that for each prime integer $p$,  the set $\{\alpha \in \widehat{Z}_p:{\ff M}_{p,\alpha}R \subsetneq R\}$ is dense with respect to the $p$-adic topology in the ring $\widehat{{\mathbb{Z}}}_p$ of $p$-adic integers  \cite[Remark 5.7]{CP}.  Here  ${\ff M}_{p,\alpha} = \{f \in $ Int$({\mathbb{Z}}):f(\alpha) \in p{\widehat{Z}}_p\}$, a maximal ideal of Int$({\mathbb{Z}})$. 
\end{remark}

Let $A$ be a P$v$MD. A {\it subintersection} of  $A$ is an overring of $A$ of the form $\bigcap_{V \in X}V$, where $X \subseteq \{A_P:P \in t${\rm-Spec}\,$A\}$. Equivalently, by Lemma~\ref{new PVMD}, a subintersection of $A$ is an intersection of critical valuation overrings of $A$.  

\begin{corollary} \label{countable Prufer}  If $A$ is a P$v$MD  such that 
$t${\rm-Spec}\,$A$ is countable, then each subintersection $B$  of $A$ can be represented   as a strongly irredundant intersection of essential valuation overrings of $A$. This representation is the only strongly irredundant representation of $B$. 
\end{corollary}

\begin{proof}
Let $X = \{A_P:P \in t${\rm-Spec}\,$A\}$, and let $Y \subseteq X$. Then  $\patch(Y)$ is a representation of $B=\bigcap_{V \in Y}V$.  Since 
 by \cite[Corollary 2.6]{FT}, $X$ is a patch closed representation of $A$, $\patch(Y) \subseteq X$, and hence $\patch(Y)$ is countable. Therefore, by (4.8), $B$ has a strongly irredundant representation in $\patch(Y)$. Since the valuation rings in $X$ are  essential for $A$, hence essential for  $B$,  Proposition~\ref{localization} implies that the valuation rings in $Y$ are critical for $B$,. Thus the ring $B$ has a strongly irredundant representation consisting of critical valuation rings. As a subintersection of 
 $A$, $B$ is a Pr\"ufer $v$-multiplication domain, hence a $v$-domain. Therefore, 
 by Theorem~\ref{v theorem}, there is only one strongly irredundant representation of $B$.   
\qed\end{proof}
 
% \begin{lemma} If $\min \X$ is Noetherian in the Zariski topology, then every integrally closed ring between $A$ and $F$ is an irredundant intersection of valuation rings. 
% \end{lemma} 
 
% What about pairwise incomparable valuation rings in $\X$? Curves say this is hard. 
 
 When $A$ is a P$v$MD for which  $t${\rm-Max}\,$A$ is a Noetherian subspace of $\Spec A$, then
$\{A_P:P \in  t${\rm-Max}\,$A\}$ is a Noetherian space of valuation overrings that represents $A$. Therefore, by Theorem~\ref{main Noetherian}, $A$ has a strongly irredundant representation in $\{A_P:P \in  t${\rm-Max}\,$A\}$. However, the fact that $A$ is a P$v$MD allows us to assert the stronger claim that $\{A_P:P \in  t${\rm-Max}\,$A\}$ itself is a strongly irredundant representation of $A$, and that this property is inherited by similarly constituted representations  of subintersections of $A$. We prove this in the next theorem.

 \begin{theorem} \label{PVMD Noetherian}  Suppose $A$ is a  P$v$MD. Then 
 %the following are equivalent. 
% \begin{description}[(3)]
 %\item[{\em (1)}]  
 $t${\rm-Max}\,$A$ (resp., $t${\rm-Spec}\,$A$) 
 is  Noetherian  in the Zariski topology
 if and only if 
% \item[{\em (2)}]  
each collection $X$ of incomparable essential valuation overrings  is a strongly (resp., tightly) irredundant representation of $\bigcap_{V \in X}V$.
%
%\item[{\em (2)}] $t${\rm-Spec}\,$A$ 
% is  Noetherian  in the Zariski topology
 %if and only if 
% \item[{\em (2)}]  
%each collection $X$ of incomparable critical valuation overrings  is a tightly irredundant representation of $\bigcap_{V \in X}V$.

 %
 %\item[{\em (3)}] $A$ satisfies $t$-$(\#\#)$.
 %\end{description} 
 \end{theorem}
 
 \begin{proof} 
 Suppose $t${\rm-Max}\,$A$  is a Noetherian space. By 
 \cite[Corollary 2.6]{FT}, the set $E = \{A_P:P \in t${\rm-Spec}\,$A\}$ of essential valuation overrings of $A$ is patch closed. Thus 
 %\cite[Remark 2.1(7)]{FT}, $t${\rm-Spec}\,$A$ is a patch closed subspace of $\Spec A$. Thus 
 by Proposition~\ref{Hoc prop}(4)(a), $E$  is spectral in the Zariski topology.  Since $\min E = \{A_P:P \in t${\rm-Max}\,$A\}$ is by assumption Noetherian in the Zariski topology, 
 Theorem~\ref{Noetherian prop} implies 
 every subset $X$ of $E$ consisting of incomparable valuation rings is discrete in the inverse topology.
 
Now let $X$ be a subset of $E$ consisting of incomparable valuation rings.
The ring  $B:=\bigcap_{V \in 
X}V$ is again a PVMD \cite[Corollary 3.9]{Kang}.  We claim that $X = \{B_P:P \in  t${\rm-Max}\,$B\}$. To this end, let $P$ be a 
 maximal $t$-ideal of $B$. Then since by Lemma~\ref{technical Noetherian}, $X$ is Noetherian and $B = \bigcap_{V \in X}V$, we have $B_P = \bigcap_{V \in X}(VB_P)$  \cite[Theorem 3.5]{ORej} and $\{VB_P:V \in X\}$ is a Noetherian subspace of $\X$ in the Zariski topology \cite[Theorem 3.7]{ORej}. Thus $\{VB_P:V \in X\}$ satisfies DCC, and since $B_P = \bigcap_{V \in X}(VB_P)$ is a valuation ring, this forces $B_P = VB_P$ for some $V \in X$, and hence $V \subseteq B_P$.  Since $V$ is essential for $A$, hence for $B$, $V = B_Q$ for some $t$-prime ideal $Q$ of $B$.  Therefore, since $P$ is  a $t$-maximal ideal and $B_Q \subseteq B_P$, we have $B_P = V \in X$, which shows $\{B_P:P \in  t${\rm-Max}\,$B\} \subseteq X$.  
 
 To verify the reverse inclusion, let $W \in X$. 
 Then since $W$ is essential,  $W = B_Q$ for some prime ideal $Q$ in $B$, and hence $Q$ is a $t$-prime ideal of $B$. Let $L$ be a maximal $t$-ideal of $B$ containing $Q$. Then, by what we have shown, $B_L \in X$, so since $B_L \subseteq W$ and the members of $X$ are incomparable, it must be that $B_L=W = B_Q$, proving that $Q \in  t${\rm-Max}\,$B$.  This proves that $X = \{B_P:P \in  t${\rm-Max}\,$B\}$. Therefore, it follows from \cite[Lemma 5, p.~119]{ZS} that $X$ is homeomorphic to $ t${\rm-Max}\,$B$, so that  $t${\rm-Max}\,$B$ is 
discrete in the inverse topology. By Theorem~\ref{Prufer} each member of $X$ is strongly irredundant in some representation of $B$. By Proposition~\ref{localization}, the valuation rings in $X$ are critical for $B$, so by (4.14), $X$ is a strongly irredundant representation of $B$.  

Conversely, suppose
 each collection $X$ of incomparable valuation rings in $\{A_P:P \in t${\rm-Spec}\,$A\}$ is a strongly irredundant representation of $\bigcap_{V \in X}V$.
Then by Theorem~\ref{Prufer} each such collection $X$ is discrete in the  inverse topology, and consequently, each collection of incomparable prime ideals in $ t${\rm-Spec}\,$A$ is discrete in the inverse  topology. Therefore, by Theorem~\ref{Noetherian prop}, 
  $t${\rm-Max}\,$A$  is  Noetherian  in the Zariski topology.
  
  Now suppose that $t${\rm-Spec}\,$A$  is Noetherian. As noted above, $E = \{A_P:P \in t${\rm-Spec}\,$A\}$ is a spectral space; it is also Noetherian since $t${\rm-Spec}\,$A$ is. Therefore, for any $V \in E$, the set $\{W \in E:V \subseteq W\}$ satisfies DCC. Thus if $X$ is any collection of incomparable valuation rings in $E$, since we have established that each valuation ring in $X$ is strongly irredundant in the intersection $\bigcap_{V\in X}V$, it follows that each valuation ring in $X$ is tightly irredundant. 
  
  Conversely, suppose that 
each collection $X$ of incomparable critical valuation overrings  is a tightly irredundant representation of $\bigcap_{V \in X}V$. We have established already that this implies that $t${\rm-Max}\,$A$ is Noetherian. Thus by Lemma~\ref{technical Noetherian}, to prove that $t${\rm-Spec}\,$A$ is Noetherian, we need only verify that for each $V \in E$, the set $\{ W \in E:V \subsetneq W\}$ has a minimal element with respect to $\subseteq$.  Let $V \in E$. Then by assumption $V$ is tightly irredundant in the representation $\{V\}$ of the ring $V$, so the intersection of the valuation rings in $\{W \in E:V \subsetneq E\}$ is again in this same set. Therefore, $t${\rm-Spec}\,$A$ is Noetherian. 
 \qed\end{proof} 
 
Gabelli, Houston and Lucas \cite{GHL} define a domain $A$ to have property $(t\#)$ if whenever $X$ is a collection of pairwise incomparable $t$-prime ideals and $P \in X$, $\bigcap_{Q \ne P}A_Q \subsetneq A_P$, where $Q$ ranges over the prime ideals in $X$ distinct from $P$. 
  Using Theorem~\ref{PVMD Noetherian}, additional characterizations of P$v$MDs  with Noetherian $t$-maximal spectrum can be deduced from the work of Gabelli, Houston and Lucas; % on the $(t\#)$ property; 
  see  for example Propositions 2.4, 2.6 and 2.8 of \cite{GHL}.  Similarly,  for a P$v$MD $A$, the characterization of Noetherian spectral spaces given in Lemma~\ref{RW lemma} can be used to link  the property  in which $t${\rm-Spec}\,$A$ is Noetherian 
  to the work of Gabelli, Houston and Lucas, specifically 
  to the equivalent characterizations in Propositions 2.11 and Theorem 2.14 in \cite{GHL}.  For additional applications, see \cite{ElB}. For example, it follows
 from Corollary~\ref{Noetherian min} and \cite[Theorem 3.9]{ElB} 
   that $A$ is a generalized Krull domain if and only if $A$ is a P$v$MD for which $t${\rm-Max}\,$A$ is Noetherian and $P\ne (P^2)_t$ for each $t$-prime ideal $P$ of $A$.

 The $(t\#)$ property extends to  non-Pr\"ufer settings the property $(\#)$ introduced for Pr\"ufer domains  by Gilmer \cite{Gilsharp} and studied further by Gilmer and Heinzer in \cite{GHsharp}.  
 A  Pr\"ufer domain $A$ is said to satisfy $(\#)$ if for each maximal ideal $M$ of $A$, $A_M$ is irredundant in the representation $\{A_N:N \in \Max A\}$ of $A$. 
Property (\#) and the stronger version $(\#\#)$, which requires that every overring has $(\#)$, play 
 an important role in the local-global theory of Pr\"ufer domains; see for example \cite{FHL, FHP}. 
 Since every maximal ideal of a Pr\"ufer domain is a $t$-maximal ideal, 
the properties $(\#)$ and $(t\#)$ coincide for Pr\"ufer domains.  Thus we have the following 
   topological characterization of Pr\"ufer domains satisfying $(\#\#)$. The corollary, which is immediate from Theorem~\ref{PVMD Noetherian}, is implicit in \cite[Theorem 5.14]{FHO}, where it is proved using the work of Rush and Wallace \cite{RW}.

\begin{corollary} Suppose $A$ is Pr\"ufer domain. Then Max $A$ is a Noetherian space if and only if $R$ satisfies $(\#\#)$. \qed
\end{corollary}

%Combining the corollary with Lemma~\ref{RW lemma}, we recover the characterization of $(\#\#)$ from {\bf [ref]}: A Pr\"ufer domain $A$ satisfies $(\#\#)$
%if and only if for each prime ideal $P$ of $A$, there is a finitely generated ideal contained in $P$ but in no maximal ideal that does not contain $P$.  

\section{Irredundance in intersections of irreducible ideals}

An ideal $A$ of the ring $R$ is {\it irreducible} if $A$ is not an intersection of two ideals properly containing it. Every ideal $A$ is the intersection of the irreducible ideals containing it. Indeed, if $r \in R \setminus A$, then by Zorn's Lemma, there exists an ideal of $R$ maximal with respect to containing $A$ and  not containing $r$. This ideal is necessarily irreducible, from which it follows that $A$ is an intersection of irreducible ideals. In this section we are interested in when $A$ is an irredundant intersection of irreducible ideals. 
We show  how the topological approach of Section 3 can be applied to intersection decompositions of irreducible ideals in {\it arithmetical rings}, those rings $R$ for which the ideals of $R_M$ form a chain for each maximal $M$ of $R$.  The ambient spectral space from which these intersection representations are drawn is $\ispec R$, the set of proper irreducible ideals of $R$, viewed as a topological space having a  basis of closed sets of the form $\V(A) := \{B \in \ispec R:A \subseteq B\}$, where $A$ ranges over all the  ideals of $R$. Since $R \not \in \ispec R$, the maximal elements in $\ispec R$ are the maximal ideals of $R$.  
%That $\ispec R$ is a spectral space is proved in the following lemma. 
By a {\it representation} of $A$ we mean a subset of $\ispec R$ whose  intersection is $A$. 

\begin{lemma} \label{arithmetical}
Let $R$ be an arithmetical ring. Then for each proper ideal $A$ of $R$, $\V(A)$ is a spectral representation of $A$ that  does not contain any proper closed representations of $A$.  
\end{lemma}

\begin{proof} 
%Statement (1) follows from the fact that since $R$ is arithmetical, an irreducible ideal $C$ of $R$ has the property that whenever $A \cap B \subseteq C$, $A \subseteq B$ or $B \subseteq C$ \cite[Lemma 2.2(3)]{HRR}. 
Every irreducible ideal in an arithmetical ring is strongly irreducible \cite[Lemma 2.2(3)]{HRR}. Also, 
the intersection of two finitely generated ideals in an arithmetical ring is finitely generated \cite[Corollary 1.11]{SW}, so by Example~\ref{spectral exs}(4), $\V(A)$ is a spectral representation of $A$. Finally, suppose $X$ is a closed subset of $\V(A)$ that is a representation of $A$. Then $X= \V(B)$ for some proper ideal $B$ of $R$. Since 
 every proper ideal of $R$ is  an intersection of irreducible ideals, $B$ is the intersection of the ideals in $\V(B)$. Since $\V(B)$ is a representation of $A$, this forces  $A  = B$. Therefore,  
 no proper closed subset of $\V(A)$ is a representation of $A$. 
%
%To see that $\ispec R$ is a tree with respect to the  specialization order, let $B,C,D \in \V(A)$ such that $D \in \V(B) \cap V(C)$.  Then $B + C \subseteq D$. Let $M$ be a maximal ideal of $R$ containing $D$. Since $B,C$ and $D$ are irreducible, we have $B = BR_M \cap R$ and $ C= CR_M \cap R$; cf. \cite[Theorem 1]{Fuc} and \cite[Lemma 1.3]{FHO}. 
%
%{\bf No}.  Need to use a primal decomposition. $\sqrt{B} + \sqrt{C} \subseteq \sqrt{D}$.  $D = DR_P \cap R$.  
%
%Minimal elements? 
%
%
%Thus since the ideals of $R_M$ form a chain and $BR_M,CR_M \subseteq DR_M$, we conclude without loss of generality that $BR_M \subseteq CR_M$, and hence $B \subseteq C$.  Thus $\ispec R$ is a tree under the specialization order. 
\qed\end{proof}

Let $A$ be a proper ideal of the arithmetical ring $R$. Since $\V(A)$ is a spectral space and the specialization order agrees with set inclusion, $\V(A)$ contains minimal elements with respect to set inclusion and $A$ is an intersection of these minimal irreducible ideals. 
%We  denote by $\min A$ the set of these minimal elements. If follows from \cite[Theorem 1]{Fuc} and \cite[Lemma 1.3]{FHO} that  
%\begin{center} $\min A = \{AR_M \cap R: M \in \Max R$ such that $A \subseteq M\}$.  
%\end{center}
Using the ideas developed in Section 3, along with the results about Noetherian spectral spaces in Section 2, we obtain a version of a theorem proved in \cite{FHO} by different methods.   
We recall  that a few notions from \cite{FHO}. 
 A  {\it Krull associated prime} of an ideal $A$ of a ring $R$   is a prime ideal that is a union of ideals of the form $A:r= \{s \in R:rs \in A\}$ with $r \in R$. 
If $P$ is a prime ideal of $R$,  we set $A_{(P)}:=\{r \in R:br \in A$ for some $b \in R \setminus P\}$.  A Zorn's Lemma argument shows that   the set of Krull associated primes of $A$ contains maximal elements. Let ${\mathcal{X}}_A$ denote the set of these maximal elements. Then (with $R$ arithmetical) we have
  $\{A_{(P)}:P \in {\mathcal{X}}_A\} = \min \V(A)$; see \cite[Theorem 5.8]{FHO}.

%We make use of the following facts. 
%\smallskip

%(a) If $A$ is an irreducible ideal of $R$, then 
% there is a prime ideal $P$ such that $A$ is {\it $P$-primal}; that is, 
 %  $A = A_{(P)}$ and $P/A$ is the set of zerodivisors of $R/A$ (\cite[Theorem 1]{Fuc} and \cite[Lemma 1.3]{FHO}).  
  % \smallskip
   
 %  (b) 
%If $R$ is arithmetical, $A$ is an ideal of $R$ and $P$ is a prime ideal, then $A_{(P)}$ is  irreducible \cite[Remark 1.6]{FHO}. 
%\smallskip
   
%   (c) 
%A prime ideal $P$ containing $A$ is a Krull associated prime  if and only if $A_{(P)}$ is $P$-primal \cite[Theorem 3.4]{FHO}.  
% \smallskip
 
%(d) By (a), (b) and (c), if $R$ is an arithmetical ring and $A$ is a proper ideal of $R$, then every proper  irreducible ideal of $R$ containing $A$ is of the form 
 % $A_{(P)}$ for some Krull associated prime $P$ of $A$.  

%(e) 

%\begin{remark}
%$\imin A$ is the set of principal components of $A$ when $\Max R$ is Noetherian.
%\end{remark}

\begin{theorem} \label{irreducible dec} 
%The following are equivalent for an arithmetical ring $R$.
%\begin{description}[(3)] 
%\item[{\em (1)}]  
{\em (cf.~\cite[Theorem 5.14]{FHO})} 
If $R$ is an arithmetical ring for which $\Max R$ is Noetherian in the Zariski topology, then for each proper ideal $A$ of $R$ the set of irreducible ideals that are minimal over $A$ is a 
 strongly irredundant representation of  $A$, and this is the unique irredundant representation of $A$ as an intersection of irreducible ideals.

% \end{description}
 
\end{theorem}

\begin{proof}  %Suppose first that $\Max R$ is a Noetherian space. 
Let $A$ be a proper ideal of $R$. As discussed before the theorem, 
 $\min \V(A) = \{A_{(P)}:P \in {\mathcal{X}}_A\} $.   
%By Lemma~\ref{arithmetical}, the set $\V(A)$  has minimal elements. 
%We claim first that each  ideal in $\min \V(A)$ is of the form $A_{(P)}:=\{r \in R:br \in A$ for some $b \in R \setminus P\}$ for some  
%Krull associated prime ideal $P$ of $A$. 
%
% Let $B \in \min \V(A)$.  Since $B$ is irreducible,  there is a prime ideal $P$ such that $B$ is $P$-primal; i.e.,   $B = B_{(P)}$ and $P/B$ is the set of zerodivisors of $R/B$ (\cite[Theorem 1]{Fuc} and \cite[Lemma 1.3]{FHO}).  Since $A \subseteq B$, we have $A_{(P)} \subseteq B_{(P)}=B$. Since $R$ is arithmetical, $A_{(P)}$ is an irreducible ideal \cite[Remark 1.6]{FHO}, and hence since $B \in \min \V(A)$, this forces $A_{(P)} = B$.  Moreover,  a prime ideal $P$ containing $A$ is a Krull associated prime  if and only if $A_{(P)}$ is $P$-primal \cite[Theorem 3.4]{FHO}.  Thus every ideal in $\min \V(A)$ is of the form $A_{(P)}$ for some Krull associated prime $P$ of $A$.  
%A Zorn's Lemma argument shows that   the set of Krull associated primes of $A$ contains maximal elements. Let ${\mathcal{X}}_A$ denote the set of these maximal elements. Then by the previous considerations,  $\{A_{(P)}:P \in {\mathcal{X}}_A\} = \min \V(A)$.  
Let $P \in {\mathcal{X}}_A$.  Then by Lemma~\ref{RW lemma} there is a finitely generated ideal $I \subseteq P$ such that $I$ is not contained in any maximal  ideal  that does not contain $P$.  Since $R$ is an arithmetical ring, the prime ideals of $R$ form a tree under inclusion, so  the only prime ideal in ${\mathcal{X}}_A$ that contains $I$ is $P$.  
 Thus since $P$ is a union of ideals of the form $A:r$, $r \in R$, there are $r_1,\ldots,r_n \in R$ such that $P$ is the only 
 ideal in ${\mathcal{X}}_A$ containing $(A:r_1) + \cdots + (A:r_n)$. Since $R$ is arithmetical, the latter ideal is equal to  $A:(r_1R \cap \cdots \cap r_nR)$ \cite[Exercise 18(c), p.~151]{LM}. Hence $r_1R \cap \cdots \cap r_nR \subseteq \left(\bigcap_{Q \ne P}A_{(Q)}\right) \setminus A_{(P)}$, where $Q$ ranges over the prime ideals in ${\mathcal{X}}_A \setminus \{P\}$.  This shows that the representation  $\min \V(A) = \{A_{(P)}:P \in {\mathcal{X}}_A\}$ of $A$ is irredundant.

%Since for each proper ideal $A$ of $R$, $R/A$ is an arithmetical ring with Noetherian maximal spectrum, we may assume without loss of generality that $A = 0$ and prove the theorem for the zero ideal of $R$. Let ${\cal M}$ denote the set of minimal irreducible ideals in $R$. Let $A  \in {\cal M}$.  Since $A$ is irreducible and $R$ is arithmetical,  $\sqrt{A}$ is a prime ideal of $R$ \cite[p.~271]{HRR}. Thus since $\Max \: R$ is Noetherian, there is by Lemma~\ref{RW lemma} (applied to the subspace $\Max R$ of $\Spec R$) a finitely generated ideal $B$ of $R$ contained in $\sqrt{A}$ such that the only maximal ideals of $R$ containing $B$ are those containing $\sqrt{A}$.  By taking a suitable power of $B$ we may assume without loss of generality that $B \subseteq A$. 
%Thus, since $\Max \: R$ coincides with the set of maximal elements of $\ispec R$ and $(\ispec R,\subseteq)$ is a tree  (Lemma~\ref{arithmetical}), we have that $B \subseteq A$ but $B$ is not contained in any other ideal in ${\cal M}$.  Therefore, ${\cal M} \setminus \{A\} \subseteq \U(B)$ while $A \not \in \U(B)$. Consequently, $A$ is isolated in the patch topology of ${\cal M}$, which proves that ${\cal M}$ is discrete in the patch topology. 

 Next, since by Lemma~\ref{arithmetical}, $\V(A)$  is a minimal closed representation of $A$, we have by Theorem~\ref{surprised} that $\min \V(A)$ is a strongly irredundant representation of $A$.  
 %assume (2). Let $C$ be a closed subset of $\ispec R$.  Then $\min C$ is the set of irreducible ideals that are minimal over $I = \bigcap_{J \in C}J$, and by (2), this representation $\min C$ of $I$ is irredundant. Thus by Lemma~\ref{ti}, $\min C$ consists of points that are isolated in the patch topology. Hence $\min C$ is discrete in the patch topology, and we conclude from Theorem~\ref{Noetherian prop} {{\bf [FIX]}} that  $\Max R$ is Noetherian in the subspace topology of $\ispec R$.   Since this topology refines the usual topology on $\Max R$,  we conclude that $\Max R$ is Noetherian 
%in the Zariski topology.  
Finally, since $\V(A)$ is a minimal closed representation of $A$, all the irreducible ideals in $\V(A)$ are critical for $A$ in the spectral representation $\V(A)$ of $A$.  Therefore, by Corollary~\ref{irredundant is si}, every irredundant representation of $A$ is strongly irredundant, and hence by Theorem~\ref{unique}(4), there is a unique irredundant representation of $A$.
\qed
\end{proof}

\begin{remark}
By \cite[Theorem 5.14]{FHO}, the converse of Theorem~\ref{irreducible dec} is also true: If every ideal $A$ of a ring $R$ can be written uniquely as an irredundant intersection of the  irreducible ideals that are minimal with respect to containing $A$, then $R$ is an arithmetical ring with Noetherian maximal spectrum. 
\end{remark}

\begin{remark} The ideas in Section 3 can also be applied to the intersections of prime ideals. Let $A$ be a radical ideal of a ring $R$. Then ${\cal V}(A) = \{P \in \Spec R:A \subseteq P\}$ is a spectral representation of $A$, each member of which is critical for $A$.  Thus $A$ has a strongly irredundant representation if and only if the set of minimal primes of $R/A$ contains a dense set of isolated points with respect to the Zariski topology (Corollary~\ref{first cor} and Theorem~\ref{unique}). Also,
every irredundant representative  of $A$ is strongly irredundant (Corollary~\ref{irredundant is si}), and 
 there is at most one irredundant representation of $A$ (Theorem~\ref{unique}). 
 Finally, if ${\cal V}(A)$ is countable, then $A$ has a strongly irredundant representation (Corollary~\ref{countable case}). 
\end{remark}

%\begin{remark}  A proper ideal $A$ of a ring $R$ is {\it completely irreducible} if it is not the intersection of its proper overideals. The Zorn's Lemma argument  at the beginning of this section shows that every proper ideal is the intersection of completely irreducible ideals.  When $R$ is an arithmetical ring, 
%\end{remark}

\end{document}